\documentclass[11 pt,leqno]{article}
\title{Symplectic structures on right-angled Artin groups: between the mapping class group and the symplectic group 
}
\author{Matthew B. Day
}
\date{July 30, 2008}
\usepackage{amsmath}
\usepackage{amsfonts}
\usepackage{amssymb}
\usepackage{amsthm}
\usepackage{latexsym}
\usepackage{epsfig}
\usepackage{xypic}
\usepackage{verbatim}
\usepackage{mathrsfs}
\usepackage{xr}
\theoremstyle{plain} \newtheorem{theorem}{Theorem}[section]
\theoremstyle{plain} \newtheorem{proposition}[theorem]{Proposition}
\theoremstyle{plain} \newtheorem{lemma}[theorem]{Lemma}
\theoremstyle{plain} \newtheorem{sublemma}[theorem]{Sublemma}
\theoremstyle{plain} \newtheorem{corollary}[theorem]{Corollary}
\theoremstyle{plain} \newtheorem{claim}[theorem]{Claim}
\theoremstyle{plain} 
\theoremstyle{plain} \newtheorem{problem}[theorem]{Open Problem}
\theoremstyle{plain} \newtheorem{conjecture}[theorem]{Conjecture}
\theoremstyle{plain} \newtheorem{question}[theorem]{Question}
\theoremstyle{remark} \newtheorem{example}[theorem]{Example}
\theoremstyle{definition} \newtheorem{remark}[theorem]{Remark}
\theoremstyle{definition} \newtheorem{definition}[theorem]{Definition}
\theoremstyle{definition} \newtheorem{theorem-definition}[theorem]{Theorem-Definition}
\theoremstyle{definition} 
\theoremstyle{plain} \newtheorem{maintheorem}{Theorem}
\theoremstyle{plain} 
\theoremstyle{plain} 

\numberwithin{equation}{subsection}

\newcommand\Aut{\mathrm{Aut}\,}
\newcommand\IAut{\mathrm{IAut}\,}
\newcommand\End{\mathrm{End}\,}
\newcommand\Out{\mathrm{Out}\,}
\newcommand\Ker{\mbox{ker}}
\newcommand\img{\mathrm{Im}\,}

\newcommand\Mod{\mathrm{Mod}}

\newcommand\Sp{\mathrm{Sp}}
\newcommand\GL{\mathrm{GL}}

\newcommand\C{\mbox{$\mathbb{C}$}}
\newcommand\Z{\mbox{$\mathbb{Z}$}}

\newcommand\into\hookrightarrow

\newcommand\isomarrow{\stackrel{\cong}{\longrightarrow}}
\def\co{\colon\thinspace}

\newcommand\badexample{
\setlength{\unitlength}{0.4 ex}
\begin{picture}(110,50)
\put(0,20){\circle*{2}}
\put(10,20){\circle*{2}}
\put(20,15){\circle*{2}}
\put(30,20){\circle*{2}}
\put(40,20){\circle*{2}}
\put(50,20){\circle*{2}}
\put(60,20){\circle*{2}}
\put(70,20){\circle*{2}}
\put(80,15){\circle*{2}}
\put(90,20){\circle*{2}}
\put(100,20){\circle*{2}}
\put(40,0){\circle*{2}}
\put(50,0){\circle*{2}}
\put(50,40){\circle*{2}}
\put(0,20){\line(1,0){10}}
\put(30,20){\line(1,0){40}}
\put(90,20){\line(1,0){10}}
\put(10,20){\line(2,-1){10}}
\put(20,15){\line(2,1){10}}
\put(70,20){\line(2,-1){10}}
\put(80,15){\line(2,1){10}}
\put(0,20){\line(5,2){50}}
\put(0,20){\line(2,-1){40}}
\put(40,0){\line(1,0){10}}
\put(40,20){\line(1,2){10}}
\put(40,20){\line(1,-2){10}}
\put(50,0){\line(1,2){10}}
\put(50,0){\line(5,2){50}}
\put(50,40){\line(5,-2){50}}
\put(50,40){\line(1,-2){10}}
\put(20,15){\line(1,0){60}}
\put(48,23){$a_1$}
\put(98,23){$b_1$}
\put(68,23){$b_2$}
\put(12,20){$a_2$}
\put(55,42){$y$}
\put(55,5){$x$}
\end{picture}
}

\newcommand\adj[2]{\mathrm e( #1,#2)}
\newcommand\AAG{\Aut A_\Gamma}
\newcommand\AAGo{\Aut^0 A_\Gamma}
\newcommand\OAG{\Out A_\Gamma}

\newcommand\supp{\mathrm{supp\,}}

\newcommand\lk{\mathrm{lk}}
\newcommand\st{\mathrm{st}}
\newcommand\pg[1]{\mathrm v(#1)}

\newcommand\lkl[1]{\lk_L(#1)}
\newcommand\stl[1]{\st_L(#1)}
\newcommand\OLS{\Omega_\ell\cup\Omega_s}

\begin{document}
\maketitle
\begin{abstract}
We define a family of groups that include the mapping class group of a genus $g$ surface with one boundary component and the integral symplectic group $\mathrm{Sp}(2g,\mathbb{Z})$. 
We then prove that these groups are finitely generated.
These groups, which we call \emph{mapping class groups over graphs}, are indexed over labeled simplicial graphs with $2g$ vertices.
The mapping class group over the graph $\Gamma$ is defined to be a subgroup of the automorphism group of the right-angled Artin group $A_\Gamma$ of $\Gamma$.
We also prove that the kernel of $\mathrm{Aut}\,A_\Gamma\to \mathrm{Aut}\,H_1(A_\Gamma)$ is finitely generated, generalizing a theorem of Magnus.
\end{abstract}

\section{Introduction}
\subsection{Background}
Let $\Gamma$ be a graph on $n$ vertices, with vertex set $X$ and adjacency relation denoted by $\adj{-}{-}$.
Let $A_\Gamma$ denote the \emph{right-angled Artin group of $\Gamma$}, defined by
\[A_\Gamma:= \langle X |R_\Gamma \rangle\]
where the relations are $R_\Gamma=\{[x,y]|\text{ $x,y\in X$ and $\adj{x}{y}$}\}$.
As we vary $\Gamma$, the group $A_\Gamma$ interpolates between the free group $F_n$ (if $\Gamma$ is edgeless) and the free abelian group $\Z^n$ (if $\Gamma$ is complete).
Similarly, as we vary $\Gamma$, the automorphism group $\AAG$ interpolates between $\Aut F_n$ and the integral general linear group $\GL(n,\Z)$.

Both mapping class groups and symplectic groups can be expressed as stabilizer subgroups of automorphism groups.
Consider the free group $F_{2g}$ with free generators $a_1,\ldots,a_g,b_1,\ldots,b_g$.
The stabilizer in $\Aut F_{2g}$ of the element $[a_1,b_1]\cdots[a_g,b_g]$ is a subgroup isomorphic to the mapping class group of a genus $g$ surface with a single boundary component.
This is a version, due to Zieschang~\cite{zieschang}, of the classical Dehn--Nielsen--Baer Theorem (see~\cite{fm}, Chapter 3).
At the other extreme, the integral symplectic group $\Sp(2g,\Z)$ is the stabilizer in $\GL(2g,\Z)$ of the standard symplectic form on $\Z^{2g}$.
In this paper, we define a structure on a right-angled Artin group that interpolates between a surface relator on a free group and a symplectic form, so that the stabilizer in $\AAG$ of this structure interpolates between mapping class groups and integral symplectic groups.
This answers a question due to Benson Farb. 

This paper is a continuation of the author's previous paper~\cite{fpraag}, and we will freely use notation, terminology, and results from that paper.

\subsection{Symplectic structures on right-angled Artin groups}
Let $A'_\Gamma=[A_\Gamma,A_\Gamma]$ be the commutator subgroup of $A_\Gamma$.
Let $H_\Gamma=A_\Gamma/A'_\Gamma$ be the abelianization of $A_\Gamma$, which is the free abelian group $\langle\{[x]|x\in X\}\rangle$.
As usual, the alternating square $\Lambda^2 H_\Gamma$ of $H_\Gamma$ is the free abelian group generated by the wedge products $[x]\wedge[y]$ for $x\neq y\in X$ (where $[x]\wedge[y]=-[y]\wedge[x]$).
The symbol $[x]\wedge[y]$ is bilinear, so the action of $\AAG$ on $H_\Gamma$ induces a diagonal action on $\Lambda^2 H_\Gamma$.

A \emph{standard alternating form} is an element of $\Lambda^2 H_\Gamma$ of the form $[a_1]\wedge[b_1]+\cdots+[a_k]\wedge[b_k]$, where $a_i, b_i\in X^{\pm1}$ and the $\{a_i,b_i\}_i$ are pairwise distinct and not equal to each other's inverses.
A \emph{surface relator} is an element (possibly trivial) of $A'_\Gamma$ of the form $[a_1,b_1]\cdots[a_k,b_k]$, where $a_i, b_i\in X^{\pm1}$, and the $\{a_i,b_i\}$ are pairwise distinct and not equal to each other's inverses.

\begin{definition}\label{de:sympstruct}
Suppose $\Gamma$ has $2g$ vertices.
A pair $(w,Q)\in A_\Gamma\times (\Lambda^2 H_\Gamma)$ is  
a \emph{symplectic structure} for the right-angled Artin group $A_\Gamma$ if 
there is some labeling of $X^{\pm1}$ as $a_1^{\pm1},\ldots,a_g^{\pm1},b_1^{\pm1},\ldots,b_g^{\pm1}$ and some $k$ with $0\leq k\leq g$ 
satisfying the following conditions:
\begin{itemize}
\item 
for each $i$ with $1\leq i\leq k$, we have $[a_i,b_i]\neq 1$,
\item
for each $i$ with $k+1\leq i\leq g$, we have $[a_i,b_i]=1$,
\item
$w$ is the surface relator
\[w=[a_1,b_1]\cdots[a_k,b_k]\]
\item 
and $Q$ is the standard alternating form:
\[Q=\sum_{i=k+1}^g[a_i]\wedge[b_i]\]
\end{itemize}

The \emph{mapping class group over $\Gamma$} with respect to a symplectic structure $(w,Q)$, written $\Mod(\Gamma,w,Q)$, is  
the intersection of the stabilizers of $w$ and $Q$ in $\AAG$:
\[\Mod(\Gamma,w,Q):=(\AAG)_{(w,Q)}=(\AAG)_w\cap(\AAG)_Q\]
\end{definition}

\begin{remark}\label{rem:spform}
Consider the subgroups $V$ and $V^\perp$ defined by:
\begin{align*} 
V&=\langle\{[x]\wedge[y]|\text{ $x,y\in X$ and $[x,y]=1$}\}\rangle\\
\hspace{2pt}
\tag*{\text{and}} 
V^\perp&=\langle\{[x]\wedge[y]|\text{ $x,y\in X$ and $[x,y]\neq 1$}\}\rangle
\end{align*}
As an $\AAG$--module, $\Lambda^2 H_\Gamma$ decomposes as
$\Lambda^2 H_\Gamma = V\oplus V^\perp$.
This phenomenon is somewhat peculiar to right-angled Artin groups.

Let $A^{(2)}_\Gamma=[A'_\Gamma,A_\Gamma]$.
The map $\Lambda^2 H_\Gamma\to A'_\Gamma/ A^{(2)}_\Gamma$ given by $[a]\wedge[b]\mapsto [a,b]\cdot A^{(2)}_\Gamma$ for $a,b\in X$ is a surjective homomorphism (it follows from the Witt-Hall identities that this map is well defined, see Serre~\cite{serlie}, Chapter II, Proposition 1.1).
Then $V$ is clearly the kernel of this map.
This also tells us that $A'_\Gamma/A^{(2)}_\Gamma\cong V^\perp$.
The composition of this isomorphism with the inclusion $V^\perp\into\Lambda^2 H_\Gamma$ gives a map 
$f\co A'_\Gamma\to \Lambda ^2 H_\Gamma$.
Note that if $[a,b]\neq 1$ for $a,b\in X$, then $f([a,b])=[a]\wedge[b]$.
This map $f$ is not usually $\AAG$--equivariant because $V^\perp$ is not usually invariant under the action of $\AAG$.

If $(w,Q)$ is a symplectic structure on $A_\Gamma$, then $Q\in V$, $f(w)\in V^\perp$ and $Q+f(w)$ is a standard symplectic form on $H_\Gamma$.
It turns out that $\AAG$ does not usually preserve $Q+f(w)$.
\end{remark}

\begin{example}
Suppose $\Gamma$ is the edgeless graph on $2g$ vertices.
Then $(w,Q)$ is a symplectic structure if and only if $Q=0$ and $w$ is a surface relator of length $4g$.
In this case $\Mod(\Gamma, w, 0)\cong\Mod_{g,1}$.
\end{example}

\begin{example}
At the other extreme, if $\Gamma$ is the complete graph on $2g$ vertices, then $(w,Q)$ is a symplectic structure if and only if $w=1$ and $Q\in V=\Lambda^2 H_\Gamma$ is a symplectic form.
In this case $\Mod(\Gamma, 1, Q)\cong\Sp(2g,\Z)$.
\end{example}

The methods of this paper make it possible to explore more difficult examples such as the following, but for brevity we give the following examples without proving the assertions we make about them.
We develop an example more thoroughly in Section~\ref{ss:counterexample} below.

\begin{example}\label{ex:join}
Suppose $\Gamma_1$ is the complete graph on $2k_1$ vertices, $\Gamma_2$ is the edgeless graph on $2k_2$ vertices, and $\Gamma$ is the graph-theoretic join of $\Gamma_1$ and $\Gamma_2$.
Then a symplectic structure on each of $A_{\Gamma_1}$ and $A_{\Gamma_2}$ will induce a symplectic structure $(w,Q)$ on $A_\Gamma$.
In this case, we have: 
\[\Mod(\Gamma,w,Q)\cong ((\Sp_{2k_1}(\Z)\times \Mod_{k_2,1})\ltimes \prod_{x\in \Gamma_2}A_{\Gamma_1})\times \Z\]
The inclusions of $\Sp_{2k_1}(\Z)$ and $\Mod_{k_2,1}$ into $\Mod(\Gamma,w,Q)$ are the obvious ones, the copy of $\Z$ is given by conjugation by $w$, and the copies of $A_{\Gamma_1}$ are given by $x\mapsto xu$ for $x$ a generator in $A_{\Gamma_2}$ and $u\in A_{\Gamma_1}$.
\end{example}

\begin{example}\label{ex:disju}
If instead of the join, we take $\Gamma$ to be the disjoint union of the graphs $\Gamma_1$ and $\Gamma_2$ from Example~\ref{ex:join}, then we still get an induced symplectic structure $(w,Q)$, but a different group $\Mod(\Gamma,w,Q)$.
Any automorphism in $\AAG$ that conjugates all of the generators of $A_{\Gamma_1}$ by one of the generators of $A_{\Gamma_2}$ (and fixes the generators of $A_{\Gamma_2}$) preserves the symplectic structure $(w,Q)$.
Although it is not obvious, it turns out that:
\[\Mod(\Gamma,w,Q)\cong\Sp_{2k_1}(\Z)\times(\Mod_{k_2,1}\ltimes A_{\Gamma_2})\]
\end{example}

In general, the groups $\Mod(\Gamma,w,Q)$ and $\Mod(\Gamma',w',Q')$ tend to look very different for slightly different graphs $\Gamma$ and $\Gamma'$.
Even if $\Gamma=\Gamma'$, if $(w,Q)$ is different from $(w',Q')$ the resulting groups may be different.
\begin{example}
If $\Gamma$ is the disjoint union of $\Gamma_1$ and $\Gamma_2$ as in Example~\ref{ex:disju}, but with a single edge added between a vertex of $\Gamma_1$ and a vertex of $\Gamma_2$, 
then the respective inclusions of $\Gamma_1$ and $\Gamma_2$ into $\Gamma$ 
do not induce maps of $\Aut A_{\Gamma_1}$ or $\Aut A_{\Gamma_2}$ to $\AAG$.
Then neither $\Mod_{k_2,1}$ nor $\Sp(2k_1,\Z)$ include in $\Mod(\Gamma,w,Q)$ in the obvious way.
However, $\Mod_{k_2-1,1}$ and $\Sp(2k_1-2,\Z)$ both include into $\Mod(\Gamma,w,Q)$, so the group is nontrivial.
It takes some work to get a more complete picture of this group.
\end{example}

\subsection{Statement of Results}
\begin{maintheorem}\label{mt:finitelygenerated}
For any graph $\Gamma$ with an even number of vertices and any symplectic structure $(w,Q)$ on $A_\Gamma$, the group $\Mod(\Gamma,w,Q)$ is finitely generated.
\end{maintheorem}
This is strong evidence that our definition for $\Mod(\Gamma,w,Q)$ from Definition~\ref{de:sympstruct} is a good one.
We also considered an alternate definition for a symplectic structure: a pair $(w,Q)$ where $w$ is a surface relator and $Q\in \Lambda^2 H_\Gamma$ is a standard symplectic form, such that $w$ and $Q$ project to the same element in $A'_\Gamma/A^{(2)}_\Gamma$.
As we show in Section~\ref{ss:counterexample}, there is an example of a graph $\Gamma$ where the subgroup of $\AAG$ fixing both a surface relator and a compatible symplectic form on $H_\Gamma$ is not finitely generated (but of course, Theorem~\ref{mt:finitelygenerated} still holds in this case).

The proof of Theorem~\ref{mt:finitelygenerated} proves both the finite generation of mapping class groups and the integral symplectic groups in special cases.
We did not find a single argument that proved both things in the same way, but rather found a single algorithm that reduces to two previously known algorithms in each extreme case.
These extremal algorithms are integral symplectic row reduction and the peak reduction algorithm (Whitehead's theorem) for free groups.

We also obtain the following statement, which is of interest in itself, as a corollary to a proposition used in the proof of Theorem~\ref{mt:finitelygenerated}.
In the case where $A_\Gamma$ is a free group, this theorem restricts to the 1934 theorem of Magnus that $\Ker(\Aut F_n\to\GL(n,\Z))$ is finitely generated (see Magnus--Karrass--Solitar~\cite{mks}, Section~3.5, Theorem~N4, or Magnus~\cite{maggitter}).
Let $\IAut A_\Gamma$ denote the kernel $\Ker(\AAG\to\Aut H_\Gamma)$.
\begin{maintheorem}\label{mt:magnusthm}
The group $\IAut A_\Gamma$ is finitely generated.
\end{maintheorem}
This theorem opens the way for further study of $\IAut A_\Gamma$.
An interesting corollary of this theorem is that the preimage in $\AAG$ of a finitely generated subgroup of $\img(\Aut A_\Gamma\to\Aut H_\Gamma)$ is a finitely generated group.

\subsection{Acknowledgements}
The results of this paper originally appeared in my Ph.D. thesis at the University of Chicago, and some of the research was done under the support of a graduate research fellowship from the National Science Foundation.
I am deeply grateful to Benson Farb, my thesis advisor, for many useful conversations and comments on earlier versions of this work.
I am grateful to Karen Vogtmann for conversations about this project and I am grateful to Ruth Charney for conversations, and for helping me find an obscure reference.
I would also like to thank Hanna Bennett, Nathan Broaddus, Thomas Church, Jim Fowler and Benjamin Schmidt for comments on earlier versions of this paper.

\section{Background}
In this section, we review the notation and the main result from Day~\cite{fpraag}.
Let $L$ be the set of letters $X\cup X^{-1}$.
For $x\in L$, let $\pg{x}\in X$, the \emph{vertex of $x$}, be the unique element of $X\cap \{x,x^{-1}\}$. 
We will use $\stl{x}$ and $\lkl{x}$ as notation for $\st(\pg{x})\cup\st(\pg{x})^{-1}$ and $\lk(\pg{x})\cup\lk(\pg{x})^{-1}$ respectively.
The \emph{support} $\supp w$ of a word or cyclic word $w$ is the subset of $X$ consisting of all generators that appear (or whose inverses appear) in $w$.
There is a reflexive and transitive binary relation on $X$ called the \emph{domination relation}: say $x\geq y$ ($x$ dominates $y$) if $\lk(y)\subset \st(x)$.
Write $x\sim y$ when $x\geq y$ and $y\geq x$; the relation $\sim$ is called the \emph{domination equivalence} relation.

There are four important classes of automorphisms known collectively as the Laurence--Servatius generators: dominated transvections, partial conjugations, inversions, and graphic automorphisms.
For $x, y\in L$ with $x\geq y$ and $\pg{x}\neq \pg{y}$, the \emph{dominated transvection} (or simply \emph{transvection}) $\tau_{x,y}$ is the automorphism that sends
$y \mapsto yx$ 
and fixes all generators not equal to $\pg{y}$.
For $x\in L$ and $Y$ a union of connected components of $\Gamma - \st(\pg{x})$, the \emph{partial conjugation} $c_{x,Y}$ is the automorphism that sends
$y \mapsto x^{-1}yx$ for $y\in Y$
and fixes all generators not in $Y$. 
For $x\in X$, the \emph{inversion} of $x$ is the automorphism that sends
$x \mapsto x^{-1}$
and fixes all other generators.
For $\pi$ an automorphism of the graph $\Gamma$, the \emph{graphic automorphism} of $\pi$ is the automorphism that sends
$x \mapsto \pi(x)$
for each generator $x\in X$.
Servatius defined these automorphisms and conjectured that they generate $\AAG$ in~\cite{se}; Laurence proved that conjecture in~\cite{la}.

We will use the \emph{Whitehead automorphisms} of $\AAG$, as defined by the author in~\cite{fpraag}.
The set of Whitehead automorphisms $\Omega$ is the finite set of all automorphisms of the following two types.
The \emph{type~(1)} Whitehead automorphisms are the finite subgroup of $\AAG$ generated by the inversions and graphic automorphisms.
An automorphism $\alpha\in\AAG$ is a \emph{type~(2)} Whitehead automorphism if there is an element $a\in L$, called the \emph{multiplier} of $\alpha$, such that for all $x\in X$, we have $\alpha(x)\in\{x,xa,a^{-1}x,a^{-1}xa\}$ (note $\alpha(a)=a$).
For $a\in L$ and $A\subset L$ with $a\in A$ and $a^{-1}\notin A$, we use the notation $(A,a)$ to refer to the type~(2) Whitehead automorphism that sends $x\in L-\{a,a^{-1}\}$ to $x$ or $a^{-1}x$ if $x\notin A$ and to $xa$ or $a^{-1}xa$ if $x\in A$, if such an automorphism exists.
Lemma~\ref{le:whdef} of Day~\cite{fpraag} explains when such an automorphism exists.


The following two subsets of $\Omega$ are also from Day~\cite{fpraag}.
The set $\Omega_\ell$ of \emph{long-range} Whitehead automorphisms is the set of all type~(1) Whitehead automorphisms together with all type~(2) Whitehead automorphisms $(A,a)$ with $A\cap\lkl{a}=\emptyset$.
The set $\Omega_s$ of \emph{short-range} Whitehead automorphisms is the set of type~(2) Whitehead automorphisms $(A,a)$ with $A\subset\stl{a}$.

We recall the definition of peak reduction.
The length of a conjugacy class in $\AAG$ is the shortest length of a representative element (with respect to $X$).
We say that a factorization $\alpha=\beta_k\cdots\beta_1$ is peak reduced with respect to a conjugacy class $[w]$ in $\AAG$ if for each $i=1,\ldots k$, we do not have both
\begin{align*}
|\beta_{i+1}\cdots\beta_1([w])|\leq|\beta_i\cdots\beta_1([w])|\\
\vspace{2pt}
\tag*{\text{and}}|\beta_i\cdots\beta_1([w])|\leq|\beta_{i-1}\cdots\beta_1([w])|
\end{align*}
unless all three lengths are equal.
We say that $\alpha$ can be peak reduced by elements of a set $S$ with respect to $[w]$ if there is a factorization $\alpha=\beta_k\cdots\beta_1$ by elements $\beta_1,\ldots,\beta_k$ that is peak reduced with respect to $[w]$.

The following theorem is essentially Theorem~\ref{mt:threeparts} of Day~\cite{fpraag}.
\begin{theorem}\label{th:threeparts}
The set $\OLS$ is a finite generating set for $\AAG$ with the following properties:
\begin{enumerate}
\item \label{it:complementss} 
each $\alpha\in\AAG$ can be written as $\alpha=\beta\gamma$ for some $\beta\in \langle\Omega_s\rangle$ and some $\gamma\in\langle\Omega_\ell\rangle$;
\item
\label{it:srinjs} 
the usual representation $\AAG\to \Aut H_1(A_\Gamma)$ to the automorphism group of the abelianization $H_1(A_\Gamma)$ of $A_\Gamma$ restricts to an embedding $\langle\Omega_s\rangle\into \Aut H_1(A_\Gamma)$; and
\item
\label{it:lrwhalgs} any  $\alpha\in\langle\Omega_\ell\rangle$ can be peak-reduced by elements of $\Omega_\ell$ with respect to any conjugacy class $[w]$ in $A_\Gamma$.
\end{enumerate}
\end{theorem}
\newcommand\partref[1]{part~(\ref{#1}) of Theorem~\ref{th:threeparts}}

We will also make use of the \emph{pure automorphism group} of $A_\Gamma$, denote $\AAGo$.
The group $\AAGo$ is the subgroup of $\AAG$ generated by the partial conjugations, dominated transvections, and inversions.
This group appears in Charney--Crisp--Vogtmann~\cite{ccv} and is useful for technical reasons.
It is easy to see that $\AAGo$ is normal in $\AAG$ and that $\AAG/\AAGo$ is finite (it is a quotient of $\Aut \Gamma$).
The group $\AAGo$ contains all those graphic automorphisms that can be realized as a product of transvections and inversions, so if $\Gamma$ is edgeless or complete, then $\AAGo$ is $\AAG$.
\section{Kernels of restrictions of the homology representation}
\label{se:KZisker}
This section is devoted to the proof of Theorem~\ref{mt:magnusthm}.
We will also prove a proposition that will be used in the proof of Theorem~\ref{mt:finitelygenerated}.

If $x,y,c\in L$ with $x,y\geq c$ and $\pg{x}$, $\pg{y}$, and $\pg{c}$ all distinct, then we write $\tau_{[x,y],c}$ for $[\tau_{x,c},\tau_{y,c}]$.
As the notation suggests, $\tau_{[x,y],c}$ sends $c\mapsto c[x,y]$ and fixes all generators in $X$ not equal to $c$.

For any subset $Z\subset X$, let $G_Z< \AAG$ be generated by the transvections $\tau_{a,b}$ for $a,b\in Z^{\pm1}$ with $a\geq b$, and the (total) conjugations of $A_\Gamma$.
Let $K_Z<\AAG$ be generated by all the $\tau_{[x,y],c}$ and all the partial conjugations $c_{x,\{c\}}$ for $x,y,c\in Z$ with $x,y\geq c$, and the (total) conjugations of $A_\Gamma$.
Note that for each $Z$, we have $K_Z< G_Z$.
We will refer to a partial conjugation of the form $c_{x,\{c\}}$ as a \emph{one-term partial conjugation}.

\begin{remark}
In fact, $K_Z$ is equal to the subgroup generated by the conjugations and the $\tau_{[x,y],c}$ and $c_{x,\{d\}}$ for $x,y,c\in Z^{\pm 1}$ and $d\in Z$ (with appropriate domination conditions).
This is because $c_{x,\{d\}}^{-1}=c_{x^{-1},\{d\}}$, and because $\tau_{[x,y],c}$ with $x,y,c\in Z^{\pm1}$ can always be expressed as a product of generators of $K_Z$.
\end{remark}

\begin{sublemma}\label{sl:conjidsconj}
For any $a,b,c,x\in X$, with $a\geq b$, $a\neq b$, $x\geq c$, and $x\neq c$,
the automorphism $\tau_{a,b}c_{x,\{c\}}\tau_{a,b}^{-1}$ is in $K_{\{a,b,c,x\}}$.

For $Y\subset X$ such that $c_{x,Y}$ is a partial conjugation of $A_\Gamma$, we have that $\tau_{a,b}c_{x,Y}\tau_{a,b}^{-1}$ is a product of elements of $K_{\{a,b,x\}}$ and partial conjugations of the form $c_{z,Y'}$ where $z\in\{a,x\}$ and $Y'\subset Y\cup\{x,a\}$.

If $c_x$ is conjugation by $x$, then $\tau_{a,b}c_x\tau_{a,b}^{-1}$ is in $K_{\{a,b,x\}}$.
\end{sublemma}
\begin{proof}
Suppose that $a,b,x\in X$, $a\geq b$, and $Y\subset X$ such that $c_{x,Y}$ is a partial conjugation.
The lemma will follow from several identities of automorphisms, which can be verified by evaluating the automorphisms on $X$.
Note that if $a=x$, then $\tau_{a,b}$ and $c_{x,Y}$ commute.
If both $a,b\in Y$, then $\tau_{a,b}$ and $c_{x,Y}$ commute.
If $a\in Y$, $b\notin Y$ and $b\neq x$, then the following identity applies:
\begin{equation}\label{eq:ainYgen}\tau_{a,b}c_{x,Y}\tau_{a,b}^{-1}=c_{x,Y}\tau_{[x,a],b}\end{equation}
The use of $\tau_{[x,a],b}$ is allowed, since if $a$ and $b$ are in different components of $\Gamma-\st(x)$ and $a\geq b$, then $x\geq b$.
If $a\in Y$ and $b=x$, then:
\begin{equation}\label{eq:ainYbisx}\tau_{a,x}c_{x,Y}\tau_{a,x}^{-1}=c_{a,(Y-a+x)}c_{x,Y}\end{equation}
These terms are allowed since if $a\geq x$ and $Y$ is a union of connected components of $\Gamma-\st(x)$, then $Y-a$ and $Y-a+x$ are both unions of connected components of $\Gamma-\st(a)$.
We have covered all the cases where $a\in Y$ or $a=x$, so we assume that $a\notin Y$ and $a\neq x$.
If both $a,b\notin Y$ and $x\neq a$ and $x\neq b$, then then $\tau_{a,b}$ and $c_{x,Y}$ commute.
If $b\in Y$, then:
\begin{equation}\label{eq:binYgen}\tau_{a,b}c_{x,Y}\tau_{a,b}^{-1}=c_{x,Y}\tau_{[x^{-1},a],b}\end{equation}
As in Equation~(\ref{eq:ainYgen}), the conditions ensure that $\tau_{[x^{-1},a],b}$ is allowed.
If $b=x$, then:
\begin{equation}\label{eq:bisxgen}\tau_{a,x}c_{x,Y}\tau_{a,x}^{-1}=c_{x,Y}c_{a,Y}\end{equation}
Since $a\geq x$ and $a\notin Y$, we have that $Y$ is a union of connected components of $\Gamma-\st(a)$ and $c_{a,Y}$ is allowed.
This proves the second statement in the lemma.

If we have some $c\in X$ with $x\geq c$, then we can take $Y=\{c\}$ and each of the Equations from~(\ref{eq:ainYgen}) through~(\ref{eq:bisxgen}) applies, proving the first statement in the lemma.

The third statement is obvious since the groups of inner automorphisms is normal in $\AAG$.
\end{proof}

\begin{sublemma}\label{sl:conjidstrans}
For any $a,b,c,x,y\in X$, with $a\geq b$, $a\neq b$, $x\geq c$, $x\neq c$, $y\geq c$, $y\neq c$, and $x\neq y$, the automorphism $\tau_{a,b}\tau_{[x,y],c}\tau_{a,b}^{-1}$ is in $K_{\{a,b,c,x,y\}}$.
\end{sublemma}
\begin{proof}
Note that
\[\tau_{[y,x],c}^{-1}=\tau_{[x,y],c}\]
so we may switch $x$ and $y$ in our enumeration of cases.

If $a,b\notin\{x,y,c\}$, then it follows from Proposition~\ref{pr:identities} of Day~\cite{fpraag} $\tau_{a,b}$ and $\tau_{[x,y],c}$ commute.
If $c=b$ then one can verify by evaluation on $X$ that:
\[\tau_{a,b}\tau_{[x,y],b}\tau_{a,b}^{-1}=c_{a,\{b\}}\tau_{[x,y],b}c_{a,\{b\}}^{-1}\]
This works whether or not $a\in\{x,y\}$.
If $c=a$ and $b\notin\{x,y\}$, then it follows from Proposition~\ref{pr:identities} of Day~\cite{fpraag} and the previous case that:
\[\tau_{a,b}\tau_{[x,y],a}\tau_{a,b}^{-1}=\tau_{[x,y],a}c_{a,\{b\}}\tau_{[y,x],b}c_{a,\{b\}}^{-1}\]
If $c=a$ and $b=x$, then 
\[\tau_{a,b}\tau_{[b,y],a}\tau_{a,b}^{-1}=c_{a,\{b\}}c_a^{-1}c_{b,\{a\}}c_b^{-1}c_{y,\{b\}}^{-1}\tau_{[y,a],b}c_{y,\{a\}}\tau_{[y^{-1},b^{-1}],a^{-1}}c_bc_{b,\{a\}}^{-1}c_ac_{a,\{b\}}^{-1}\]
where $c_a$ and $c_b$ denote the (total) conjugations by $a$ and $b$ respectively.
Our assumptions dictate that $y\geq a\sim b$, so all the terms in this equation are allowed.
This identity can be verified by evaluation on $X$.

Now we may assume that $c\notin\{a,b\}$.
If $b=x$ and $a\neq y$, then the following identity applies:
\[\tau_{a,b}\tau_{[b,y],c}\tau_{a,b}^{-1}=c_{b,\{c\}}\tau_{[a,y],c}c_{b,\{c\}}^{-1}\tau_{[b,y],c}\]
Again, this identity can be verified by evaluation.
There are then two remaining cases: $a=x$ and $b\neq y$;  and $a=x$ and $b=y$.
In both of these cases, it follows from Proposition~\ref{pr:identities} of Day~\cite{fpraag} that $\tau_{a,b}$ commutes with $\tau_{[x,y],c}$.
\end{proof}

\begin{lemma}\label{le:KZnormal}
For any $Z\subset X$, the group $K_Z$ is normal in $G_Z$.
\end{lemma}

\begin{proof}
If $a,b\in X$ with $a\geq b$, then $\tau_{a^{-1},b}=\tau_{a,b}^{-1}$, and $\tau_{a,b^{-1}}\tau_{a,b}=c_{a,\{b\}}$.
This means that $G_Z$ is generated by the generators of $K_Z$ together with the transvections $\tau_{a,b}$ with $a,b\in K$ (in particular, not in $K^{-1}$).
Then the identities from Sublemma~\ref{sl:conjidsconj} and Sublemma~\ref{sl:conjidstrans} indicate that the conjugate of any generator of $K_Z$ by a generator of $G_Z$ can be expressed as a product of elements of $K_Z$.
\end{proof}

The proof of the following proposition is a generalization of Magnus's proof that $IA_n$ is finitely generated~\cite{maggitter}.

\begin{proposition}\label{pr:KZisker}
For any $Z\subset X$, we have $K_Z=\Ker(G_Z\to \Aut H_\Gamma)$.
\end{proposition}

\begin{proof}
Let $C_1\cup\cdots\cup C_m=Z$ be the decomposition of $Z$ into domination equivalence classes.
Since partial conjugations map to the identity in $\Aut H_\Gamma$, it follows from Corollary~\ref{co:tvgenstruct} of Day~\cite{fpraag} that $\img(G_Z\to\Aut H_\Gamma)$ has a presentation where the generators are the elementary row operations $E_{a,b}=(\tau_{a,b})_*$ such that $a\geq b$, for $a,b\in Z$, and the relations are as follows:
%
%
\begin{enumerate}
\item\label{it:rr1}
$[E_{a,b},E_{c,d}]=1$ if $b\neq c$ and $a\neq b$,
\item\label{it:rr2}
$[E_{a,b},E_{b,d}]E_{a,d}^{-1}=1$ if $a\neq d$,
\item\label{it:rr3}
$(E_{a,b}E_{b,a}^{-1}E_{a,b})^4=1$, if $a\sim b$ and $a\neq b$,
\item\label{it:rr4}
$(E_{a,b}E_{b,a}^{-1}E_{a,b})^2(E_{a,b}E_{b,a}^{-1}E_{a,b}E_{b,a})^{-3}=1$, if $a,b\in C_i$, $a\neq b$ and $|C_i|=2$.
\end{enumerate}

Consider the lifts of the relations gotten by replacing each of the $E_{a,b}$ with the corresponding $\tau_{a,b}$.
We claim that these lifts are all in $K_Z$.
Relation~(\ref{it:rr1}) obviously lifts to $\tau_{[a,c],b}$ if $b=d$ and lifts to the trivial element otherwise.
Relation~(\ref{it:rr2}) lifts to $[\tau_{a,b},\tau_{b,d}]\tau_{a,d}^{-1}$, which is $\tau_{[b,a],d}$.
We know $K_Z$ is normal in $G_Z$, so we say two elements of $G_Z$ are equal modulo $K_Z$ if their difference is in $K_Z$.
Since $\tau_{a,b^{-1}}\tau_{a,b}=c_{a,\{b\}}$, we know that $\tau_{a,b}^{-1}$ and $\tau_{a,b^{-1}}$ are equal modulo $K_Z$.
Then the lift the element $E_{a,b}E_{b,a}^{-1}E_{a,b}$ is equal, modulo $K_Z$, to $\tau_{a,b^{-1}}^{-1}\tau_{b,a^{-1}}\tau_{a,b}$, which is equal to the permutation $\sigma_{a,b}$ of order $4$ from Equation~(\ref{eq:R5}) of Day~\cite{fpraag}, according to that equation.
So relation~(\ref{it:rr3}) lifts to an element of $K_Z$.
The lift of the element $(E_{a,b}E_{b,a}^{-1}E_{a,b}E_{b,a})^3$ is equal modulo $K_Z$ to $(\sigma_{a,b}\tau_{b,a})^{3}$.
By Equation~(\ref{eq:R6}) of Day~\cite{fpraag}, $(\sigma_{a,b}\tau_{b,a})^3=\tau_{a,b}^{-1}\tau_{b,a^{-1}}^{-1}\tau_{a,b^{-1}}\sigma_{a,b}^{3}$.
This is equal modulo $K_Z$ to $\tau_{a,b}^{-1}\tau_{b,a^{-1}}^{-1}\tau_{a,b}^{-1}\sigma_{a,b}^{3}$, which is $\sigma^2$ by Equation~(\ref{eq:R5}) of Day~\cite{fpraag}.
So relation~(\ref{it:rr4}) lifts to an element of $K_Z$.

The group $K_Z$ is obviously in $\Ker(G_Z\to \Aut H_\Gamma)$.
Any element of $G_Z$ can be expressed as a product of inner automorphisms, one-term partial conjugations, and lifts $\{\tau_{a,b}|a,b\in Z, a\geq b\}$ of the $\{E_{a,b}|a,b\in Z, a\geq b\}$.
Since these lifts map to the generators of our presentation for $\img(G_Z\to \Aut H_\Gamma)$ and the inner automorphisms and one-term partial conjugations are in $\Ker(G_Z\to\Aut H_\Gamma)$, 
it follows that any element of $\Ker(G_Z\to \Aut H_\Gamma)$ can be written as a product of conjugates of inner automorphisms, one-term partial conjugations, and lifts of relators from the presentation.
The group $K_Z$ contains all the inner automorphisms, one-term partial conjugations, and lifts of the relators. 
By Lemma~\ref{le:KZnormal}, $K_Z$ is normal in $G_Z$, so it contains all the conjugates of these elements.
So $\Ker(G_Z\to \Aut H_\Gamma)<K_Z$, and they are equal.
\end{proof}

Recall from the introduction that $\IAut A_\Gamma$ denotes the kernel of the homology representation.
We will show Theorem~\ref{mt:magnusthm} by showing that $\IAut A_\Gamma$ is generated by the generators of $K_X$, together with the partial conjugations of $A_\Gamma$.
\begin{proof}[Proof of Theorem~\ref{mt:magnusthm}]
Let $\rho\co\AAG\to \Aut H_\Gamma$ be the homology representation.
As previously noted, $\AAGo$ is normal in $\AAG$.
It is apparent from considering the generators of $\AAG$ and the definition of $\AAGo$ that $\rho$ induces an isomorphism $\AAG/\AAGo\cong \rho(\AAG)/\rho(\AAGo)$.
From this we deduce that $\IAut A_\Gamma<\AAGo$.

Let $K$ be the subgroup of $\AAGo$ generated by $K_X$ and the partial conjugations.
Note that $\AAGo$ is generated by $G_X$ together with $K$ and the inversion automorphisms. 

By Sublemma~\ref{sl:conjidsconj}, Lemma~\ref{le:KZnormal} and the fact that inversions normalize $K$, we know that $K$ is normal in $\AAGo$.
So if $\alpha\in \AAGo$, then $\alpha$ can be written as $\alpha= \beta\gamma$ where $\beta\in K$ and  $\gamma$ is a product of elements of $G_X$ and inversions.
If we further assume that $\alpha$ is in $\IAut A_\Gamma$, then it follows from Proposition~\ref{pr:KZisker} that $\gamma$ is in $K_X$.
So $\IAut A_\Gamma< K$.
Since the reverse inclusion is obvious, it follows that $\IAut A_\Gamma=K$ and $\IAut A_\Gamma$ is generated by the finite set of the generators of $K_X$ together with the partial conjugations.
\end{proof}

\section{Symplectic structures}
\subsection{A counterexample}\label{ss:counterexample}

As an alternate definition for a symplectic structure on a right-angled Artin group, one can consider a pair $(w,\widetilde Q)$ where $w\in A'_\Gamma$ is a surface relator, $\widetilde Q\in \Lambda^2 H_\Gamma$ is a symplectic form, and $w$ and $\widetilde Q$ map to the same element under the respective maps of $A'_\Gamma$ and $\Lambda^2 H_\Gamma$ to $A'_\Gamma/A^{(2)}_\Gamma$.
The group 
$(\AAG)_{(w,\widetilde Q)}$ can also be seen as an analogue to a mapping class group or a symplectic group. 
This differs from Definition~\ref{de:sympstruct} in that $\widetilde Q$ is a symplectic form on all of $\Lambda^2 H_\Gamma$, instead of being an alternating form supported on a subspace.

This alternate definition is attractive because the groups defined in this way have symplectic homology representations, while in general the groups $\Mod(\Gamma,w,Q)$ do not.
However, this alternate definition is less attractive because of the following example, which is a group that satisfies the alternate definition and is not finitely generated.

\begin{figure}[htb]
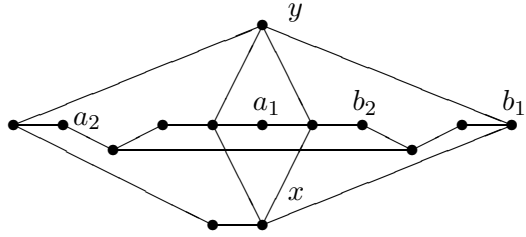

\begin{center}
\badexample
\end{center}
\caption{A counterexample to the finite generation of a different group.}\label{fi:badex}
\end{figure}

\begin{example}\label{ex:counterexample}
Take $\Gamma$ to be the graph indicated in Figure~\ref{fi:badex}.
Let $w$ be the word $[a_1,b_1][a_2,b_2]$, where $a_1$, $b_1$, $a_2$ and $b_2$ are as indicated.
By pairing off the remaining vertices in adjacent pairs $a_3,b_3,\ldots, a_7,b_7$ (which include the vertices labeled $x$ and $y$), we get a symplectic form:
\[\widetilde Q=\sum_{i=1}^7 [a_i]\wedge[b_i]\]
It is easy to see that $(w,\widetilde Q)$ satisfies the alternate definition.

One can check that the only examples of domination in this $\Gamma$ are $x\geq a_1$ and $y\geq a_1$, and that the only vertices whose stars separate $\Gamma$ are $x$ and $y$, both of which separate it into $\{a_1\}$ and one other component.
By inspecting the valences of the vertices, it is apparent that there are no nontrivial automorphisms of the graph $\Gamma$.

From Laurence's theorem (Theorem~\ref{th:lau}), we can tell that $\AAG$ is generated by conjugations, inversions, and the following four transvections:
\[\{\tau_{x,a_1},\tau_{y,a_1},\tau_{x,a_1^{-1}},\tau_{y,a_1^{-1}}\}\]
Note that $\tau_{x,a_1}(a_1)=a_1x$ and $\tau_{y,a_1}(a_1)=a_1y$, while $\tau_{x,a_1^{-1}}(a_1)=x^{-1}a_1$ and $\tau_{y,a_1^{-1}}(a_1)=y^{-1}a_1$.
Let $F_2$ denote the free group on the generators $x$ and $y$.
If $\alpha$ is in the subgroup generated by these four transvections, then $\alpha(a_1)=u^{-1}a_1v$ for some $u,v\in F_2$.
The map $\alpha\mapsto (u,v)$ is an isomorphism from this subgroup to $F_2\times F_2$.
Then we have
\begin{equation}\label{eq:outex}\OAG\cong (\Z/2\Z)^{14}\ltimes (F_2\times F_2)\end{equation}
where the fourteen generators of order 2 are the inversions and the inversions act on the transvections by the rule of Equation~(\ref{eq:R6}).

If $(u,v)\in F_2\times F_2$, then the corresponding outer automorphism sends the conjugacy class $[w]$ to the class represented by the graphically reduced cyclic word $u^{-1}a_1b_1a_1^{-1}ub_1^{-1}[a_2,b_2]$.
The $v$ does not appear because $x$ and $y$ both commute with $b_1$.
It then follows from Equation~(\ref{eq:outex}) that $(\OAG)_{[w]}$ is the subgroup generated by the images of
$\tau_{x,a_1}$, $\tau_{y,a_1}$ together with the inversions of vertices that do not appear in $w$.
The only inner automorphisms fixing $w$ are conjugation by powers of $w$.
At this point, we can see that
\[(\AAG)_w\cong ((\Z/2\Z)^{10}\ltimes F_2)\times \Z\]
where the copy of $F_2$ is generated by $\tau_{x,a_1}$ and $\tau_{y,a_1}$ and the copy of $\Z$ is generated by conjugation by $w$.

In the subgroup of $\Aut H_\Gamma$ generated by $(\tau_{x,a_1})_*$ and $(\tau_{y,a_1})_*$, it is easy to check that only the trivial element preserves $\widetilde Q$.
Then it follows from Proposition~\ref{pr:KZisker} (with $Z=\{x,y,a_1\}$) that the group $(\AAG)_{\widetilde Q}$ is also finitely generated.
We do not see any groups that are not finitely generated until we stabilize both $w$ and $\widetilde Q$.

As before, if $\alpha\in \langle \tau_{x,a_1},\tau_{y,a_1}\rangle$, then $\alpha(a_1)=a_1u$, where $u\in\langle x,y\rangle\cong F_2$ and the map $\alpha\mapsto u$ is an isomorphism.
The subgroup of $\langle\tau_{x,a_1},\tau_{y,a_1}\rangle$ fixing $Q$ is then isomorphic to the kernel of the abelianization map $F_2\to \Z^2$.
The only products of inversions preserving both $w$ and $\widetilde Q$ are $N_{a_i}N_{b_i}$ for $i=3,\ldots,7$ and their products, where $N_{z}$ denotes the inversion with respect to $z$.
We can then deduce that:
\[(\AAG)_{(w,\widetilde Q)}\cong ((\Z/2\Z)^5\ltimes (\Ker(F_2\to \Z^2)))\times \Z\]
Since $\Ker(F_2\to\Z^2)$ is an infinite rank free group, we have that $(\AAG)_{(w,\widetilde Q)}$ is not finitely generated.

On the other hand, if we take $Q$ to be $\widetilde Q$ minus the image of $w$ in $\Lambda^2 H_\Gamma$, then 
\[\Mod(\Gamma,w,Q)\cong ((\Z/2\Z)^5\ltimes F_2)\times \Z,\]
which is finitely generated.
\end{example}
\begin{remark}
This example shows that the image of $\Mod(\Gamma,w,Q)$ under the homology representation $\AAG\to\Aut H_\Gamma$ need not be symplectic.
Although $\Mod(\Gamma,w,Q)$ fixes $w$ and fixes $Q$, it doesn't necessarily fix $\widetilde Q=f(w)+Q$, where $f:A'_\Gamma\to\Lambda^2 H_\Gamma$ is as in Remark~\ref{rem:spform}.
This is because the map $f$ is not usually equivariant. 
\end{remark}

\subsection{Symplectic row reduction with domination}
At this point, we assume that $|X|=2g$ is even.
Pick a bijection $()^*:L\to L$ such that $(a^*)^*=a^{-1}$ for all $a\in L$, and pick a set of $g$ letters $S=\{a_1,\ldots,a_g\}\subset L$ such that $S^*\cup S$ contains $x$ or $x^{-1}$ for each $x\in X$.
Let
\[Q=\sum\{[a_i]\wedge[a_i^*]|a_i^*\in\lkl{a_i}\}\in\Lambda^2H_\Gamma\]
and let $w_0$ be the concatenation of the words $[a_i,a_i^*]$ for those $i$ for which $a_i^*\notin\lkl{a_i}$, in increasing order of the index $i$.
Then $(w_0, Q)$ satisfies the definition of a symplectic structure on $A_\Gamma$.
We will also demand that there is some $k$, $0\leq k\leq g+1$, such that $[a_i,a_i^*]\neq1$ for $i=1,\ldots, k$ and $[a_i,a_i^*]=1$ for $i=k+1,\ldots,g$.
In this subsection, we assume $Q\neq 0$.
Let $\supp Q\subset X$ denote the set of elements $a\in X$ with $a$ appearing in $Q$.

In this subsection, for $a$ in $L$, we will also use $a$ to denote the image of $a$ in $H_\Gamma$.
The images of the elements of $X$ give a basis for $H_\Gamma$ which we also call $X$.
By declaring $X$ to be orthonormal, we determine an inner product $\langle-,-\rangle:H_\Gamma\times H_\Gamma\to\Z$.
For $a,b\in L$ with $a\geq b$, let $E_{a,b}\in \Aut H_\Gamma$ denote the linear transvection (the row operation) mapping
\[b\mapsto b+a\] 
and fixing the images of all elements of $X$ different from $\pg{a}$. 
For $a\in L$, let $N_a\in\Aut H_\Gamma$ denote the inversion with respect to $a$, which maps 
\[a\mapsto -a\]
and fixes the images of all elements of $X$ different from $\pg{a}$.

A \emph{standard dominated $Q$--transvection} (or \emph{$Q$--transvection} for short) is an element of $\Aut H_\Gamma$ of one of the two following forms:
\begin{enumerate}
\item $E_{a,a^*}$, where $a\in\supp Q$ and $a\geq a^*$,
\item $E_{a,b}E_{b^*,a^*}^{-1}$ where $a,b\in\supp Q$, $\pg{a^*}\neq\pg{b}$, $a\geq b$ and $b^*\geq a^*$.
\end{enumerate} 
The \emph{$Q$--inversion} with respect to $a\in\supp Q$ is $N_aN_{a^*}$.
Note that a standard $Q$--transvection is not necessarily a transvection, but it is in some sense the closest thing to a transvection that preserves $Q$.
In the case that $\Gamma$ is a complete graph, the standard dominated $Q$--transvections are simply the standard symplectic transvections from classical linear algebra.

\begin{claim}
The $Q$--transvections and $Q$--inversions preserve $Q$.
\end{claim}
\begin{proof}
Note the following computations:
\[N_aN_{a^*}\cdot (a\wedge a^*)=(-a)\wedge(-a^*)=a\wedge a^*\]
\[E_{a,a^*}\cdot (a\wedge a^*)=a\wedge(a^*+a)=a\wedge a^*\]
\begin{align*}
E_{a,b}E_{b^*,a^*}^{-1}\cdot(a\wedge a^*+b\wedge b^* )&=
E_{a,b}\cdot(a\wedge a^*+b\wedge b^*-a\wedge b^*)\\
&=a\wedge a^*+b\wedge b^*
\end{align*}
The claim follows immediately.
\end{proof}

Let $G<\Aut H_\Gamma$ be the group:
\[G=\langle\{E_{a,b}|\text{$a\in X$, $b\in \supp Q$, and $a\geq b$}\}\cup\{N_b|b\in\supp Q\}\rangle\]
This is the image under the homology representation of the subgroup of $\AAGo$ that fixes each element of $(\supp w_0)^{\pm1}$.

This subsection is devoted to the proof of the following theorem.
\begin{theorem}\label{th:domspred}
The subgroup of $G$ stabilizing $Q$ is generated by the standard dominated $Q$--transvections and $Q$--inversions.
\end{theorem}

The basis $X$ lets us think of $\End H_\Gamma$ as matrices; in particular, it allows us to identify $\End H_\Gamma$ with $\otimes^2 H_\Gamma$, and gives us a transpose operation.
Express $\End H_\Gamma$ as three-by-three block matrices, with respect to the decomposition of $H_\Gamma$ as: 
\[\langle\supp w_0\rangle\oplus\langle a_{k+1},\ldots,a_g\rangle\oplus\langle a_{k+1}^*,\ldots,a_g^*\rangle\]
Define $J\in\End H_\Gamma$ by:
\[J:=\left(
\begin{array}{ccc}
0&0&0\\
0&0&-I_{g-k}\\
0&I_{g-k}&0
\end{array}
\right)\]
Then $J$ is the image of $Q$ under the map:
\[\Lambda^2 H_\Gamma \into \otimes^2 H_\Gamma \isomarrow \End H_\Gamma\]
Note that for any $A\in \Aut H_\Gamma$, we have $A\cdot Q=Q$ if and only if $AJA^\top=J$.
Also note that for any $a\in\supp Q$, we have $Ja=a^*$.
Let $H_Q<H_\Gamma$ be generated by the image of $\supp Q$.  

\begin{lemma}\label{le:justHQ}
If $A_0\in G$ and $A_0$ fixes $Q$, then $A_0$ leaves $H_Q$ invariant.
\end{lemma}
\begin{proof}
It follows from the definition of $G$ that for some matrices $A$, $B$, $C$, $D$, $E$, $F$, we have:
\[A_0=\left(
\begin{array}{ccc}
I_{2k}&E&F\\
0&A&B\\
0&C&D
\end{array}
\right)\]
Since $A_0JA_0^\top=J$, we can deduce that:
\[\left(\begin{array}{cc}A&B\\C&D\end{array}\right)\cdot
\left(\begin{array}{c}-F^\top\\E^\top\end{array}\right)=
\left(\begin{array}{c}0\\0\end{array}\right)\]
But since the matrix $\binom{A\ B}{C\ D}$ is a diagonal block of a block-upper-triangular matrix, it is invertible, and therefore $F=0$ and $E=0$.
\end{proof}

By virtue of Lemma~\ref{le:justHQ}, we restrict our entire argument from $H_\Gamma$ to $H_Q$.
We also use the symbols $Q$ and $J$ to represent their respective restrictions to $H_Q$.
Note that $J$ restricted to $H_Q$ is invertible.
The element $Q$ is a standard symplectic form, when considered as an element of $\Lambda^2 H_Q$.

\begin{lemma}\label{le:knowdom}
If $\alpha\in \AAGo$, then for any $a,b\in X$, we have
$\langle \alpha_*b,a\rangle\neq 0$ implies $a=b$ or $a\geq b$.
\end{lemma}

\begin{proof}
We induct on the length of $\alpha$ in terms of the generators of $\AAGo$.
The assertion is obvious if the length is zero.
Assume it is true for $\beta$ and that $\alpha=\beta\gamma$ where $\gamma$ is one of the generators of $\AAGo$.
If $\gamma$ is an inversion or a partial conjugation, then lemma follows.
Suppose $\gamma=\tau_{c,d}$ for some $c,d\in L$ with $c\geq d$.
Then $\langle \alpha_*b,a\rangle\neq 0$ implies either that
$\langle \beta_*b,a\rangle\neq 0$ or that $\pg{a}=\pg{c}$ and $\langle \beta_*b,d\rangle\neq 0$.
In the first case, the lemma follows.
In the second case, we have $a\geq d$ and also $d\geq b$ by inductive assumption.
\end{proof}
\newcommand\qgeq{\;\dot{\geq}\;}
\newcommand\qsim{\;\dot{\sim}\;}

We define a binary relation on $\supp Q$ called \emph{$Q$--domination}: $a\qgeq b$ if $\pg{a}\neq\pg{b^*}$ and $a\geq b$ and $b^*\geq a^*$, or if $\pg{a}=\pg{b^*}$ and $a\geq b$.
We define \emph{$Q$--domination equivalence} by: $a\qsim b$ if $a\qgeq b$ and $b\qgeq a$.
Note that we have a $Q$--transvection sending $b$ to $b+a$ only if $a\qgeq b$.

\begin{lemma}\label{le:knowQdom}
If $A\in G$ and $A$ fixes $Q$, then for any $a,b\in \supp Q$, we have that
$\langle Ab,a\rangle\neq 0$ implies $a=b$ or $a\qgeq b$.
\end{lemma}

\begin{proof}
If $\langle Ab,a\rangle\neq 0$, then by Lemma~\ref{le:knowdom}, we have $a\geq b$.
Since $AJA^\top=J$, we know $A=J(A^\top)^{-1}J^{-1}=J^\top(A^{-1})^\top J$.
So 
\[0\neq \langle J^\top(A^{-1})^\top Jb,a\rangle=\langle b^*,A^{-1}a^*\rangle\]
which implies (again by Lemma~\ref{le:knowdom}) that $b^*\geq a^*$.
\end{proof}

Now we will reassign the indices for our basis for $H_Q$.
Assume we have labeled some vertices $\{x_1,\ldots,x_i,y_1,\ldots,y_i\}\subset(\supp Q)^{\pm1}$.
Then we choose $x_{i+1}$ to be $Q$--domination maximal among the elements of $\supp Q$ not yet labeled as $x_j$ or $y_j$.
Set $y_{i+1}=Jx_{i+1}$.
By construction, we deduce that $\pg{y_{i+1}}$ is not in $\{\pg{x_1},\ldots,\pg{x_i},\pg{y_1},\ldots,\pg{y_i}\}$.
We proceed this way until we have constructed a basis.

We will now prove Theorem~\ref{th:domspred} by exhibiting a row-reduction algorithm.
This algorithm will differ from the usual integral symplectic row-reduction algorithm in that we have to check at each step that the $Q$--domination relation allows us to use a given $Q$--transvection. 
\begin{proof}[Proof of Theorem~\ref{th:domspred}]
Let $A\in G$ fix $Q$.  
By Lemma~\ref{le:justHQ}, we think of $A$ as being in $\Aut H_Q$.
Assume inductively that we have already row-reduced $A$ by applying standard $Q$--transvections and inversions to get a matrix $A_i$ (for $1\leq i\leq g-k$) such that for $j<i$, we have:
\begin{equation}\label{eq:matids1}A_ix_j=x_j\end{equation}

Since $A$ is symplectic, $A_i$ is symplectic, so  $A_i=J^\top(A_i^{-1})^\top J$, and
\begin{equation}\label{eq:matids2}
\langle A_ix_i,y_j\rangle=-\langle y_i,A_i^{-1}x_j\rangle=-\langle y_i,x_j\rangle=0\end{equation}
for any $j<i$.
If we have any $j\geq i$ with $\langle A_i x_i,x_j\rangle\neq 0$, then by Lemma~\ref{le:knowQdom}, we have $x_j\qgeq x_i$; 
since $x_i$ is maximal among $\{x_i,\ldots, x_{g-k}\}$, we know $x_i\qsim x_j$.
Similarly, if $j\geq i$ with $\langle A_i x_i,y_j\rangle\neq 0$, we know $x_i\dot\sim y_j$.

{\bf Step 1:} We consider all the indices $j\geq i$ such that both $\langle A_i x_i,x_j\rangle\neq 0$ and $\langle A_i x_i,y_j\rangle\neq 0$.
For any such $j$, we have $x_j\dot\sim y_j$, and by repeatedly applying the row operations $E_{x_j,y_j}$, $E_{y_j,x_j}$ and their inverses (which are $Q$--transvections) according to the Euclidean algorithm, we can reduce $A_i$ to a matrix $\hat A_i$ in which either $\langle \hat A_i x_i,x_j\rangle= 0$ or $\langle \hat A_i x_i,y_j\rangle= 0$.
By repeating this step for all such indices $j$, we assume we have reduced $A$ to $A_i'$ in which for each $j\geq i$, either $\langle A'_i x_i,x_j\rangle= 0$ or $\langle A'_i x_i,y_j\rangle= 0$.
Note that these operations do not affect the columns of $A_i$ before column of $x_i$, so equations~(\ref{eq:matids1}) and~(\ref{eq:matids2}) both still hold with $A_i'$ instead of $A_i$.

{\bf Step 2:} We find an element $a\in\{x_i,\ldots,x_{g-k},y_i,\ldots,y_{g-k}\}$ which maximizes $|\langle A'_i x_i,a\rangle|$ for $a$ in this set.
Since the determinant of $A'_i$ is nonzero, we can deduce from the form of $A'_i$ that this maximum is nonzero.
If this $a$ is the unique such element making this value nonzero, we move on to the next step.
Otherwise, there is some other $b\in\{x_i,\ldots,x_{g-k},y_i,\ldots,y_{g-k}\}$ with $|\langle A'_i x_i,b\rangle|\neq 0$.
Since these matrix entries are nonzero, we know that $a\dot\sim b$.
Since by the first step, we know that $\langle A'_i x_i,a^*\rangle=\langle A'_i x_i,b^*\rangle=0$, we know that the row operation $E_{b,a}E_{b^*,a^*}^{-1}$ and its inverse only change the column of $x_i$ in $A'_i$ by adding plus or minus the $b$--entry to the $a$--entry.
Further, this does not alter the column of $x_j$ in $A'_i$ for any $j<i$.
This step reduces either the maximum of $|\langle A'_i x_i,a\rangle|$ for $a\in\{x_i,\ldots,x_{g-k},y_i,\ldots,y_{g-k}\}$, or it reduces the number of elements realizing this maximum absolute value.
Either way, by repeatedly applying this step, we arrive at a matrix $A_i''$ such that there is a unique $a\in \{x_i,\ldots,x_{g-k},y_i,\ldots,y_{g-k}\}$ with $|\langle A''_i x_i,b\rangle|\neq 0$.
Again, the equations~(\ref{eq:matids1}) and~(\ref{eq:matids2}) both still hold with $A_i''$ instead of $A_i$.

{\bf Step 3:} 
We now have a unique $a\in\{x_i,\ldots,x_{g-k},y_i,\ldots,y_{g-k}\}$ with $\langle A''_i x_i,a\rangle\neq 0$.
By the form of $A_i''$ and the fact that its determinant is $1$, we deduce that
$|\langle A''_i x_i,a\rangle|=1$.

If $\pg{a}\neq \pg{x_i}$, then we know $a\qsim x_i$.
In this case, if $\pg{a}\neq\pg{y_i}$ we can apply the product of $Q$--transvections
\[(E_{a,x_i}E_{x_i^*,a^*}^{-1})(E_{x_i,a}^{-1}E_{a^*,x_i^*})(E_{a,x_i}E_{x_i^*,a^*}^{-1})\]
which sends $x_i$ to $a$, $a$ to $-x_i$, $y_i$ to $a^*$, and $a^*$ to $-y_i$ while fixing all other elements of our basis.
Otherwise,  $\pg{a}=\pg{y_i}$ and we can apply the product of $Q$--transvections
\[E_{y_i,x_i}E_{x_i,y_i}^{-1}E_{y_i,x_i}\]
which sends $x_i$ to $y_i$ and $y_i$ to $x_i^{-1}$ while fixing all other elements of our basis.
In any event, perhaps by applying some $Q$--transvections, we may assume that $\pg{a}=\pg{x_i}$.
Possibly after applying a $Q$--inversion, we may assume that
$\langle A''_i x_i,x_i\rangle=1$.
So for $j\geq i$, we have $\langle A''_ix_i,y_j\rangle=0$ and $\langle A''_ix_i,x_j\rangle$ is $0$ if $i\neq j$ and $1$ if $i=j$.

{\bf Step 4:}
For each $j<i$ with $\langle A''_i x_i,x_j\rangle\neq 0$, we know that $x_j\qgeq x_i$, and we may apply the row operation $E_{x_j,x_i}E_{y_i,y_j}^{-1}$.
Since all of the other relevant entries in the column of $x_i$ are zero, the only effect of this operation is to add $1$ to the $x_j$--entry.
Of course, by applying some power of this operation, we can delete this entry.
By applying this step repeatedly, we arrive at a matrix that satisfies the recursion hypotheses for $i+1$ and we can go back to step 1.

We recurse through these steps $g-k$ times and arrive at a matrix $A_{g-k+1}$ satisfying Equations~(\ref{eq:matids1}) and~(\ref{eq:matids2}) for $i=g-k+1$.
However, these conditions imply that $A_{g-k+1}$ is the identity matrix.

Since we reduced the arbitrary $A$ to the identity matrix by repeatedly applying $Q$--transvections and $Q$--inversions, we have shown that these elements generate the stabilizer of $Q$ in $G$.
\end{proof}

\subsection{Automorphisms fixing a surface relator}
Recall the bijection $*:L\to L$ with $(a^*)^*=a^{-1}$ for all $a\in L$, and the surface relator:
\[w_0=[a_1,b_1]\cdots[a_k,b_k]\]
Note that $|w_0|=4k$.
In this subsection, we assume that $|w_0|>0$.
Note that if $a\in\supp w_0$, then either $(a^*)^{-1}aa^*$ or $a^*a^{-1}(a^*)^{-1}$ is a subsegment of $w_0$.

From here on, we will use symbols like $w$ to refer to a word or the group element it determines, and we will use $[w]$ to refer to the cyclic word determined by $w$ or to the conjugacy class of $w$.

The goal of this subsection is to prove that we can peak-reduce an arbitrary automorphism in $\AAG$ (not just in $\langle \Omega_\ell\rangle$) if we are only reducing peaks with respect to $[w_0]$ (Theorem~\ref{th:prw0}).
In order to do this, we split an automorphism fixing $[w_0]$ into its long-range and short-range parts, and we will analyze this short-range part (Sublemma~\ref{sl:moveirrel} and Lemma~\ref{le:specialsrform}).
Once we understand the short-range part, we will be able to absorb all of the peak-forming short-range elements into general Whitehead automorphisms (Lemma~\ref{le:prone}).

We define the \emph{pure long-range Whitehead automorphisms} $\Omega_\ell^0$ to be $\Omega_\ell \cap\AAGo$.
We will use a slight refinement of \partref{it:lrwhalgs}: if $[w]$ is a conjugacy class and $\alpha\in\langle\Omega_\ell^0\rangle$, then $\alpha$ can be peak-reduced with respect to $[w]$ by elements of $\Omega_\ell^0$ (see Remark~\ref{re:purewhalg} of Day~\cite{fpraag}).

Our first goal is to show that the action of $\AAGo$ cannot shorten $[w_0]$ or shrink its support.
This relies strongly on the structure of $w_0$, which we exploit through the following two sublemmas. 

\begin{sublemma}\label{sl:hitCform}
Suppose $C$ is a nonempty adjacent domination equivalence class of $\Gamma$ and suppose $\gamma$ is a product of transvections and partial conjugations with multipliers in $C$.
Let $j=|C\cap\supp w_0|$ and let $m=(k-j)$.
There are letters $f_1,\ldots,f_j\in (\supp w_0\cap C)^{\pm1}$, $x_1,\ldots,x_j$, $c_1,\ldots,c_m$ and $d_1,\ldots,d_m$ in $L-C^{\pm1}$ and words $u_i=\gamma(f_i)$, $r_1,\ldots,r_m$, $s_1,\ldots,s_m$ and $t_1,\ldots,t_m$ in $C^{\pm1}$ such that $\gamma([w_0])$ is represented by a cyclic word given as a graphically reduced product of the words
\[x_1u_1x_1^{-1},\ldots,x_ju_jx_j^{-1}\]
with the words
\[c_1r_1d_1s_1c_1^{-1}t_1d_1^{-1},\ldots,c_mr_md_ms_mc_m^{-1}t_md_m^{-1}\]
and some elements of $C^{\pm1}$ in some order.
Further, if $j>0$ then these $\{u_i\}_i$ are all nontrivial, and their product is nontrivial.
\end{sublemma}

\begin{proof}
We will prove this statement by induction on the length of $\gamma$ as a product of Whitehead automorphisms.
First we discuss the base case.
For a factor $[a_i,b_i]$ of $w_0$, not both of $a_i$ and $b_i$ are in $C$ since $a_i$ and $b_i$ do not commute.
If a given $a_i$ is in $C^{\pm1}$, we set the next available $x_p=b_i$ and set $u_p=f_p=a_i^{-1}$.
Similarly, if $b_i$ is in $C^{\pm1}$, we set the next available $x_p=a_i$ and set $u_p=f_p=b_i$.
For each $i$ with both $\pg{a_i},\pg{b_i}\notin C$, we set the next available $c_p=a_i$ and set $d_p=b_i$.
We take each $r_p$, $s_p$, and $t_p$ to be the empty word.
This proves the base case $|\gamma|=0$.

Now suppose that $w$ satisfies the conclusions of the lemma and let $\alpha$ be a Whitehead automorphism with multiplier $a\in C^{\pm1}$, so that $\gamma=\alpha\gamma'$ for some $\gamma'$.
Then for each element $x$ of $X-C$, $\alpha(x)$ contains a single instance of $x$, and $\supp\alpha(x)\subset \{\pg{a},x\}$.
Then the same choices of $\{x_i\}_i$, $\{c_i\}_i$ and $\{d_i\}_i$ elements will work.
For each $i$, it is possible to choose new words $r_i$, $s_i$ and $t_i$ that will work based on the old words and $\alpha(c_i)$ and $\alpha(d_i)$.
Note that regardless of what $\alpha(x_i)$ is, $\alpha(x_iu_ix_i^{-1})$ is $x_i\alpha(u_i)x_i^{-1}$ or $a^{-1}x_i\alpha(u_i)x_i^{-1}a$.
This means that our $\alpha(u_i)$ will work as our new $u_i$, and since $u_i=\gamma'(f_i)$, we have $\alpha(u_i)=\gamma(f_i)$.
This means that we can write $\alpha(w)$ in the desired form.

Note that since each original $f_i$ is nontrivial, each $\gamma(f_i)$ is nontrivial, and since the product $f_1\cdots f_j$ is nontrivial, the product $\gamma(f_1)\cdots\gamma(f_j)$ is nontrivial.
\end{proof}
\begin{sublemma}\label{sl:shortenwithC}
Suppose $C$ is an adjacent domination equivalence class of $\Gamma$ with $|C|>1$, and suppose $\gamma$ is a product of transvections and partial conjugations with multipliers in $C$.
Then if $\alpha\in\Omega_\ell$ shortens $\gamma([w_0])$, then $\alpha=(A,a)$ for some $a\in C^{\pm1}$.

Further, no $\alpha\in\Omega_\ell$ can shorten $[w_0]$.
\end{sublemma}

\begin{proof}
We will prove both statements at once by supposing that either $C=\emptyset$ (and $\gamma$ is trivial) or $|C|>1$.
If $\alpha\in\Omega_\ell$ shortens $\gamma([w_0])$, then $\alpha$ is not a permutation automorphism, so suppose $\alpha=(A,a)$.
Suppose for contradiction that $\pg{a}\notin C$.
Let $w$ be a representative of $\gamma([w_0])$ of the form given in Sublemma~\ref{sl:hitCform} (or $w=w_0$ if $C=\emptyset$).
The conjugacy class of $w$ maps to the trivial element of $\Aut H_\Gamma$.
This means that every element of $\supp w$ appears an even number of times in $w$, half with positive exponent and half with negative exponent.
Since $\pg{a}\notin C$, we know by Sublemma~\ref{sl:hitCform} that $a$ appears only twice in $w$.
So we have $v_1, v_2$ words in $L-\{a,a^{-1}\}$ such that $w=v_1av_2a^{-1}$ as graphically reduced cyclic words.
Then since $\alpha$ shortens $w$, it must delete both the instance of $a$ and the instance of $a^{-1}$ in $w$ without introducing any new instances of $a^{\pm1}$.
Then $\alpha(v_1)=v_1$ and $\alpha(v_2)=a^{-1}v_2a$ (since $\alpha$ multiplies generators by $a$ only on the right).

In the case $|C|>1$, we have distinct $b, c\in C^{\pm1}$.
Suppose $a\geq b$.
Then $a\in\lkl{c}$, and therefore $a\in\lkl{b}$.
So either $a$ commutes with every element of $C$ or $a$ does not dominate any element of $C$. 

With notation as in Sublemma~\ref{sl:hitCform}, 
we first suppose that $\pg{a}$ is a $\pg{c_i}$ or a $\pg{d_i}$.
This is the only possibility if $C=\emptyset$.
We suppose that $\pg{a}=\pg{c_i}$, since the case that $\pg{a}=\pg{d_i}$ is parallel.
Then possibly by swapping $(A,a)$ with $(L-A-\lkl{a},a^{-1})$, we assume that $a=c_i$.
Then $v_2=r_id_is_i$.
If $C=\emptyset$, of course, our words $r_i$, $s_i$ and $t_i$ are all empty.
If $r_i\neq 1$ or $s_i\neq 1$ and $a$ does not commute with the elements of $C$, then to delete the instances of $a^{\pm1}$ in $ar_i$ and $s_ia^{-1}$, $\alpha$ must conjugate $C$.
Then in any event, either to delete existing instances or to avoid introducing new instances of $a^{\pm1}$, $d_i$ and $d_i^{-1}$ must be in $A$ (note that $d_i\notin\lkl{a}$ since $a=c_i$).
We know that $t_id_i^{-1}$ is an initial segment of $v_1$.
If $t_i=1$, then we already have a contradiction, since $d_i$ is then the first letter of $v_1$ and we have $\alpha(v_1)\neq v_1$.
If $t_i\neq1$, $a$ does not commute with $t_i$, and to avoid introducing an extra instance of $a$ between $t_i$ and $d_i^{-1}$, $\alpha$ must conjugate $C$.
But then $\alpha(t_i)=a^{-1}t_ia^{-1}$, and we cannot have $\alpha(v_1)=v_1$, which is a contradiction.
So $\pg{a}$ is not a $\pg{c_i}$ or a $\pg{d_i}$.
Note that in the case $C=\emptyset$, the proof is finished.

Then it must be that $\pg{a}=\pg{x_i}$ for some $i$.
For the rest of the proof, we assume $|C|>1$.
We suppose without loss of generality that $a=x_i$.
Then $v_2=u_i$, a word in $C^{\pm1}$.
If $a$ commutes with the elements of $C$, then since $\alpha$ is long-range, $\alpha$ fixes $u_i$, and we have $\alpha(v_2)=v_2$, which is a contradiction.
So suppose $a$ does not dominate any element of $C$.
Then $\alpha$ sends every element of $C$ to its conjugate by $a$.
Since $u_i$ is nontrivial and $w$ maps to the trivial element of $H_\Gamma$, we know that there are some elements of $C$ in $v_1$.
So there is a subsegment $v_3$ of $v_1$ such that $\alpha(v_3)=a^{-1}v_3a$.
Let $v_4$ be the longest subsegment of $v_1$, containing $v_3$, such that $\alpha(v_4)= a^{-1}v_4a$.
Since $\alpha(v_1)=v_1$, we know that if we delete $v_4$ from $v_1$ we get two subsegments.

By Sublemma~\ref{sl:hitCform}, the letter furthest to the left in this right subsegment of $v_1$ must be an element of $C^{\pm1}$, or an $x_i$, or a $c_i$.
If it is an element of $C^{\pm1}$, then it is conjugated by $\alpha$, contradicting the definition of $v_4$.
If it is an $x_i$, then this $x_i$ is in a subsegment $x_iu_ix_i^{-1}$.
Since $v_4$ maps to $a^{-1}v_4a$, we know that $x_i^{-1}$ must be in $A$ or else that $x_i\in\lkl{a}$.
If $x_i\notin\lkl{a}$, then the $x_i^{-1}$ on the right adds an instance of $a$, contradicting the definition of $v_4$.
If $x_i\in\lkl{a}$, then since the $u_i$ is nontrivial, it is conjugated by $a$, also contradiction the definition of $v_4$.
If this letter is a $c_i$, there are several cases.
If $c_i,d_i\in\lkl{a}$, then either the $a$ from $\alpha(v_4)$ commutes past our $c_ir_id_is_ic_i^{-1}t_id_i^{-1}$, or one of the $r_i$, $s_i$ or $t_i$ is nontrivial and an $a$ is introduced by conjugation.
If $d_i$ is in $\lkl{a}$ but $c_i$ is not, then to avoid introducing an $a$, we have $c_i^{-1}\in A$ and another $a$ is introduced either by the $c_i^{-1}$ or the $t_i$.
If $c_i$ is in $\lkl{a}$ but $d_i$ is not, then either the $a$ from $v_4$ or from $r_i$ must be cancelled by an $a^{-1}$ from $d_i$, so $d_i^{-1}\in A$ and the final $d_i^{-1}$ introduces an extra $a$.
If both $c_i,d_i\notin\lkl{a}$, then to cancel the $a$ from $v_4$, $c_i^{-1}\in A$; to cancel the $a$ from $c_i^{-1}$ or $t_i$, we have $d_i\in A$; to cancel the $a$ from $d_i$ or from $s_i$, we have $c_i\in A$; and to cancel the $a$ from $c_i$ or $r_i$, we have $d_i^{-1}\in A$.
This means that $d_i^{-1}$ introduces an extra $a$ at the end.
In any event we contradict the definition of $v_4$ if the letter in $v_1$ right after $v_4$ is a $c_i$.
So $v_4$ must extend to the right edge of $v_1$, a contradiction. 
\end{proof}

\begin{lemma}\label{le:suppmin}
Suppose $C$ is the domination equivalence class of an element $c\in X$.
Then if $\gamma\in\AAGo$ we have:
\[|C\cap \supp \gamma([w_0])|\geq |C\cap\supp w_0|\] 
\end{lemma}

\begin{proof}
By Theorem~\ref{th:threeparts}, we can write $\gamma=\alpha\beta$ where $\alpha\in\langle\Omega_\ell^0\rangle$ and $\beta\in\langle\Omega_s\rangle$.
By Theorem~\ref{th:threeparts}, we can write $\beta=\beta'\beta_C$,
where $\beta_C$ is a product of short-range transvections with multipliers in $C$ and $\beta'$ is a product of short-range transvections with multipliers not in $C$.
Again by Theorem~\ref{th:threeparts}, we can rewrite $\alpha\beta'$ as $\beta''\alpha'$ for some $\alpha'\in\langle\Omega_\ell^0\rangle$ and some $\beta''\in\langle\Omega_s\rangle$.
Further, by the form of the sorting substitutions in Definition~\ref{de:sort} of Day~\cite{fpraag} (in the proof of part~\ref{it:complements} of Theorem~\ref{mt:threeparts} of Day~\cite{fpraag}), we know that $\beta''$ is also a product of short-range transvections with multipliers not in $C$.

We have $\gamma=\beta''\alpha'\beta_C$.
The automorphism $\beta''$ cannot remove any instances of elements of $C$ from the support of a word because it can be written as a product of transvections whose multipliers are not in $C$.
We will prove the lemma by showing that $\beta_C$ cannot send $w_0$ to a word containing fewer elements of $C$, and then by showing that 
$\alpha'$ cannot remove elements of $C$ from $\beta_C([w_0])$.

Consider $(\beta_C)_*\in\Aut H_\Gamma$ as a matrix with respect to the generators of $H_\Gamma$ given by the image of $X$.
By Corollary~\ref{co:tvgenstruct} of Day~\cite{fpraag}, we know that the block of this matrix sending the image of $C$ to itself is invertible. 
Then:
 \[\left|\bigcup\{C\cap \supp \beta_C(a)|a\in C\cap \supp w_0\}\right|\geq \left|C\cap\supp w_0\right|\]
Suppose temporarily that $C$ is an adjacent domination equivalence class.
If $\pg{a}\in C\cap\supp w_0$, we know that $\pg{a^*}\notin C$ since $a^*\notin\lkl{a}$.
Then $\pg{a^*}\notin\supp\beta_C(x)$ for any $x\neq (a^*)^{\pm1}$, and since $a^*\notin\lkl{a}$, we also know $\beta_C(a^*)=a^*$.
If $C$ is a non-adjacent domination equivalence class, then $\beta_C=1$, and it follows in both cases that $\beta_C(a^*)=a^*$ for every $a$ with $\pg{a}\in C\cap \supp w_0$.

Consider the representative $w$ for $\beta_C([w_0])$ from Sublemma~\ref{sl:hitCform}.
For each $a\in C\cap\supp w_0$, $\pg{a^*}$ is one of the $\pg{x_i}$ elements and $\beta_C(a)$ is the corresponding $u_i$.
In particular, there are at least as many elements of $C$ appearing in subsegments of $w$ of the form $(a^*)^{-1}\beta_C(a)a^*$ or $a^*\beta_C(a)^{-1}(a^*)^{-1}$, for $a\in C^{\pm1}$ and $a^*\notin C^{\pm1}$ as there are elements of $C\cap \supp w_0$.

By Theorem~\ref{th:threeparts}, $\alpha'$ has a factorization by elements of $\Omega^0_\ell$ that is peak-reduced with respect to $\beta_C([w_0])$.
This factorization may include permutations, but these permutations preserve $C$ (because they are in $\AAGo$), so the only way to remove any extra instances of elements of $C$ from $w$ is to decrease its length.
Then peak reduction implies that the elements of this factorization shorten $w$ immediately and keep shortening it until all the excess instances of elements of $C$ have been removed.
If $C$ is an adjacent domination equivalence class and $|C|>1$, then by Sublemma~\ref{sl:shortenwithC}, each one of these shortening automorphisms has multiplier in $C$, and again by Sublemma~\ref{sl:hitCform}, we see that these shortening automorphisms do not remove any elements from $C\cap\supp w_0$.
If $|C|=1$ or $C$ is a non-adjacent domination equivalence class, then $\beta_C=1$ and $\beta_C([w_0])=[w_0]$.
Then by Sublemma~\ref{sl:shortenwithC}, no element of $\Omega_\ell$ can shorten $\beta_C([w_0])$, and therefore $\alpha'$ cannot remove any elements from $C\cap\supp w_0$.
\end{proof}

\begin{corollary}\label{co:lenmin}
For each $\gamma\in\AAGo$, we have: 
\[|\gamma([w_0])|\geq |w_0|\]
\end{corollary}

\begin{proof}
We know from Lemma~\ref{le:suppmin} that $|\supp\gamma([w_0])|\geq 2k$. 
Since $\gamma([w_0])$ maps to the trivial element of $H_\Gamma$, each element appears at least twice (once with positive and once with negative exponent).
So $|\gamma([w_0])|$ is at least $4k$, the length of $w_0$.
\end{proof}

Now we will analyze the short-range part of an automorphism fixing $[w_0]$.

\begin{definition}
A short-range transvection $\tau_{a,b}$ is \emph{$w_0$--irrelevant} if $a\in\lkl{b^*}$.
An automorphism is \emph{$w_0$--irrelevant} if it is a product of $w_0$--irrelevant transvections.
\end{definition}

\begin{remark}
Suppose $\adj{a}{b^*}$ and $a\geq b$.
Then:
\[\tau_{a,b}(bb^*b^{-1}(b^*)^{-1})=bab^*a^{-1}b^{-1}(b^*)^{-1}=bb^*b^{-1}(b^*)^{-1}\]
So $\tau_{a,b}$ fixes $w_0$.
We call these automorphisms $w_0$--irrelevant because they are an obvious class of automorphisms fixing $w_0$.
\end{remark}

\begin{sublemma}\label{sl:moveirrel}
Let $C$ be an adjacent domination equivalence class in $X$.
Suppose $\alpha\in\langle\Omega^0_\ell\rangle$, $\sigma$ is a permutation automorphism that fixes $C$, $\gamma$ is a product of short-range transvections with multipliers not in $C$, and $\beta$ is a product of short-range transvections with multipliers in $C$, such that:
\[\gamma\sigma\alpha\beta([w_0])=[w_0]\]
Then we can write $\beta$ as $\beta'\iota$, where $\iota$ is $w_0$--irrelevant and $\beta'$ is a product of short-range transvections with multipliers in $C$, none of which are $w_0$--irrelevant.
\end{sublemma}

\begin{proof}
First we note that the group
\[\langle \{\tau_{a,b}|\pg{b}\in\supp w_0\text{ and } a\in C\cap \lkl{b}\cap\lkl{b^*}\}\rangle\]
is a normal subgroup of the group generated by short-range transvections with multipliers in $C$.
If $b\in\supp w_0$ and $a\in C\cap\lkl{b}\cap\lkl{b^*}$, then $b\notin C$.
If $b$ were in $C$, then $b\sim a$ and $b\in\lkl{b^*}$, a contradiction.
So for $\tau_{c,d}$ a short-range transvection with $c\in C$, we have $\pg{c}\neq\pg{b}$, and either $\tau_{c,d}$ commutes with $\tau_{a,b}$, or $\pg{d}=\pg{a}$ and we apply $[\tau_{c,a},\tau_{a,b}]=\tau_{c,b}$.
Since $c\geq a$, we have $c\in\lkl{b}$, and $\tau_{c,b}$ is a member of the subgroup and the subgroup is normal.

So we can move $w_0$--irrelevant transvections $\tau_{a,b}$ with $\pg{b}\in\supp w_0$ to the right of any other transvections in a factorization of $\beta$, and therefore without loss of generality we may assume that $\beta$ has a factorization in which the only $w_0$-irrelevant transvections that appear are the ones of the form $\tau_{a,b}$ where $\pg{b}\notin\supp w_0$.

Apply $\beta$ to $w_0$ letter-by-letter and graphically reduce to get a cyclic word $w$.
Then $w$ is a representative of $\beta([w_0])$.
Suppose there is some $b\in \supp w_0$ and $y\in C -\supp w_0$ with $y\in \supp \beta(b)$.
By the form of $\beta$, we know that no element in $\supp \beta(b)$ commutes with $b^*$.  
Since $b^*$ does not commute with $b$, we know $b^*$ is not in $C$, and therefore each instance of $b^*$ survives in $w$.
No instance of $y$ can be cancelled out of the subsegment $b^*\beta(b)^{-1}(b^*)^{-1}$ or $(b^*)^{-1}\beta(b)b^*$ of $w$ (this is the image of the subsegment $b^*b^{-1}(b^*)^{-1}$ or $(b^*)^{-1}bb^*$ of $w_0$).

Since $y\notin\supp w_0$, either $\alpha$ or $\gamma$ or $\sigma$ must remove it.
We know that $\sigma$ fixes $C$, so $\sigma$ cannot remove it.
Also $\gamma$ cannot remove $y$ because $\gamma$ can be written as a product of transvections whose multipliers are not $y$.
We can peak-reduce $\alpha$ with respect to $w$.
This peak-reduced factorization may have permutation automorphisms in it, but these will fix adjacent domination equivalence classes.
So, there must be a sequence of long-range automorphisms, each of which progressively shortens $w$, which remove all instances of $y$.
This is impossible: the $b^*$ and $(b^*)^{-1}$ cannot be removed since removing one of them would change the class of the word in $H_\Gamma$ and removing both would contradict Lemma~\ref{le:suppmin}; they cannot be moved without being removed since this would not shorten the word; and without moving or removing the $b^*$ and $(b^*)^{-1}$ it is impossible to remove the instance of $y$ between them.
This is a contradiction, so we may assume that for $b\in \supp w_0$, we have $\supp \beta(b)\subset\supp w_0$.

This fact, together with Theorem~\ref{th:threeparts}, lets us deduce that $\beta$ has a factorization by short-range transvections with multipliers in $C\cap \supp w_0$.
Note that the subgroup 
\[\langle\{\tau_{x,y}|x\in C\cap \supp w_0\text{ and }y\notin\supp w_0\}\rangle\]
is normal in the group of short-range automorphisms with multipliers in $C\cap \supp w_0$.
This is because for any $\tau_{x,y}$ with $x\in C\cap\supp w_0$ and $y\notin \supp w_0$, and any $\tau_{a,b}$ with $a\in C\cap\supp w_0$,
either $\tau_{x,y}$ and $\tau_{a,b}$ commute or $\pg{b}=\pg{x}$ and we apply the identity
$[\tau_{a,b},\tau_{b,x}]=\tau_{a,x}$.
Since this subgroup is normal, we can rewrite $\beta$ with all the $w_0$--irrelevant transvections first.
\end{proof}

The following lemma is the core reason that we are able to peak-reduce automorphisms fixing $[w_0]$, regardless of whether they are long-range or not.
\begin{lemma}\label{le:specialsrform}
Suppose $\alpha\in\langle\Omega^0_\ell\rangle$ and $\beta\in\langle\Omega_s\rangle$ such that $\alpha\beta([w_0])=[w_0]$.
Then there is a permutation automorphism $\sigma$ that leaves $\supp w_0$ invariant, a $w_0$--irrelevant automorphism $\iota\in\langle\Omega_s\rangle$, distinct elements $x_1,\ldots, x_r\in (\supp w_0)^{\pm1}$ (with $x_i\neq x_j^{-1}$ for any $i,j$) and some elements $y_1,\ldots, y_r\in (\supp w_0)^{\pm1}$ with $x_i\geq y_i$ and $y_i\in\lkl{x_i}$ such that
\[\beta=\sigma\tau_{x_1,y_1}\cdots\tau_{x_r,y_r}\iota\]
and such that each $x_i$ is domination-minimal among $\{x_i,x_{i+1},\ldots,x_r\}$.
\end{lemma}

\begin{proof}
Suppose that $C_1\cup\cdots\cup C_m=X$ is the decomposition of $X$ into adjacent domination equivalence classes.
We assume that these sets are indexed such that if $a\in C_i$ and $b\in C_j$ with $a\in\lkl{b}$, $a\geq b$ and $a\not\sim b$, then $i>j$.  
This assumption makes $C_1$ minimal and makes $C_m$ maximal.

Inductively assume we have expressed $\beta$ as
\[\beta'\sigma_{p-1}\delta\iota_{p-1}\]
where $\beta'$ is a product of short-range transvections whose multipliers are in $C_{p}\cup \ldots\cup C_m$; 
the automorphism $\delta$ can be written as a product of short-range transvections with distinct multipliers in $C_1\cup\ldots\cup C_{p-1}$, in domination-nondecreasing order; 
the automorphism $\iota_{p-1}$ is $w_0$--irrelevant; 
and $\sigma_{p-1}$ is a permutation automorphism that is trivial outside of $C_1\cup\ldots\cup C_{p-1}$.  
We will show that we can then do the same for $p$ instead of $p-1$.

By Corollary~\ref{co:tvgenstruct} of Day~\cite{fpraag}, we can rewrite $\beta'$ as 
$\beta''\beta_p$, where $\beta_p$ is a product of short-range transvections whose multipliers are in $C_p$ and $\beta''$ is a product of short-range transvections whose multipliers are in $C_{p+1}\cup\ldots\cup C_m$.
Then $\beta_p$ commutes with $\sigma_{p-1}$.
We can also conjugate $\beta_p$ across $\delta$, as follows.
Observe that if we have short-range transvections $\tau_{a,b}$ and $\tau_{c,d}$ with $c\in C_p$ and $a\in C_i$ for $i<p$, then the transvections commute unless $\pg{d}=\pg{a}$, in which case we have
$[\tau_{c,a},\tau_{a,b}]=\tau_{c,b}$.
In any case, we do not change $\delta$ by conjugating these elements across it, and the new transvections we introduce have multipliers in $C_p$.
As a result we can write $\beta$ as
\[\beta''\sigma_{p-1}\delta\beta'_p\iota_{p-1}\]
where $\beta'_p$ is a product of transvections whose multipliers are in $C_p$.

Next we move $\beta''$ back across $\alpha$ and move $\delta$ across $\sigma_{p-1}$ and $\alpha$.
Of course, this is possible by Theorem~\ref{th:threeparts}, but we also note that by Equation~(\ref{eq:adjsubs}) of Day~\cite{fpraag}, if we introduce new short-range transvections through this process, they will have the same multipliers as those already in $\beta''$ and $\delta_{p-1}$.
So we can write
\[\alpha\beta''\sigma_{p-1}\delta=\gamma\sigma_{p-1}\alpha'\]
where $\gamma$ is a product of short-range transvections with multipliers not in $C_p$ and $\alpha'\in\langle\Omega^0_\ell\rangle$.

Since $\iota_{p-1}$ is $w_0$--irrelevant, one can easily see that $\iota_{p-1}([w_0])=[w_0]$.
Then since $\alpha\beta=\gamma\sigma_{p-1}\alpha'\beta'_p\iota_p$, we have $\gamma\sigma_{p-1}\alpha'\beta'_p([w_0])=[w_0]$.
By Sublemma~\ref{sl:moveirrel}, we can write $\beta'_p\iota_{p-1}$ as $\beta''_p\iota_p$, where $\beta''_p$ is a product of transvections 
with multiplier in $C_p$ that are not $w_0$--irrelevant, and $\iota_p$ is a product of $w_0$--irrelevant transvections.
In particular, we have $\beta=\beta''\sigma_{p-1}\delta\beta''_p\iota_p$.
 
If we consider $(\beta''_p)_*\in\Aut(H_\Gamma)$ as a matrix, we know that the block of $(\beta''_p)_*$ taking the image of $C_p$ to itself is invertible.
Further, since we have removed all the $w_0$--irrelevant automorphisms, we know that the block of $(\beta''_p)_*$ taking the image of $C_p\cap \supp w_0$ to itself is invertible.  
Then there is a permutation $\sigma'_p$ of $(C_p\cap\supp w_0)^{\pm1}$ such that $\sigma'_p(x)$ appears in $\beta''_p(x)$ to a positive power, for $x\in X$.
We extend $\sigma'_p$ by the identity outside of $C_p$ to get a permutation of $L$; since $C_p$ is an adjacent domination equivalence class this permutation extends to an automorphism of $A_\Gamma$.
Let $\delta_p=(\sigma'_p)^{-1}\beta''_p$, and let $\sigma_p=\sigma_{p-1}\sigma'_p$.
Then each $x\in X$ appears in $\delta_p(x)$ to a positive power.
Note that the hypotheses on $\delta$ imply that $\delta$ commutes with $\sigma'_p$, and we have
$\beta=\beta''\sigma_p\delta\delta_p\iota_p$.

Let $w$ be a word gotten by applying $\delta_p$ letter-by-letter to $w_0$ and graphically reducing.
Since $\delta_p$ is free of $w_0$-irrelevant transvections in its factorization, for any $a\in\supp w_0$, we know $\supp \delta_p(a)$ does not contain any elements commuting with $a^*$.
Then if we further suppose that $\supp\delta_p(a)\neq\{a\}$, then $a$ is adjacently dominated by an element of $C_p$, and we know that $a^*\notin C_p$ and we have $a^*\delta_p(a)^{-1}(a^*)^{-1}$ or $(a^*)^{-1}\delta_p(a)a^*$ as a subsegment in $w$.
This is also true if $a\in C_p$.
Of course, $\alpha\beta''\sigma_p\delta([w])=[w_0]$.
If we let $\alpha''=(\sigma'_p)^{-1}\alpha'\sigma'_p\in\langle\Omega_\ell\rangle$, then $\alpha\beta''\sigma_p\delta=\gamma\sigma_p\alpha''$.
We know that $\gamma$ cannot remove any instances of elements in $C_p$ from a word, and $\sigma_p$ can permute the elements of $C_p\cap\supp w_0$ but cannot remove any.
Therefore if the elements of $\supp\delta_p(a)$ in $w$ are removed by $\gamma\sigma_p\alpha''$, it must be $\alpha''$ that removes them.
We assume $\alpha''$ to be peak-reduced with respect to $\delta_p([w_0])$, so there must be a sequence of long-range automorphisms that progressively shortens $\delta_p([w_0])$ and remove the extra instances of elements of $C_p$.
However, we know we cannot alter the instance of $a^*\delta_p(a)^{-1}(a^*)^{-1}$ or $(a^*)^{-1}\delta_p(a)a^*$ in $w$ by any such moves.
If on the other hand $\delta_p(a)=a$, we know by Lemma~\ref{le:suppmin} that $\sigma_p(a)$ survives to the final $w_0$.

Therefore for each element $c$ appearing in $\delta_p(a)$ for any $a\in\supp w_0$, the element $\sigma_p(c)$ appears in the final $w_0$ with at least the multiplicity with which $c$ appears in $\delta_p(a)$.  
We know $\sigma_p(c)\in(\supp w_0)^{\pm1}$ if and only if $c\in(\supp w_0)^{\pm1}$.
Therefore there cannot be any $a\in \supp w_0$ with $\delta_p(a)$ containing any $x\in \supp w_0$ to any power greater than $1$ in absolute value, or with $\delta_p(a)$ containing any $x\notin\supp w_0$ at all.
Finally, if there are two distinct elements $c_1, c_2\in \supp w_0$ and some $x$ with $x\in\supp \delta_p(c_i)$ for $i=1,2$, then $x$ appears in $c_i^*\delta_p(c_i)^{-1}(c_i^*)^{-1}$ or $(c_i^*)^{-1}\delta_p(c_i)c_i^*$ for $i=1,2$, and also in $x^*\delta_p(x)^{-1}(x^*)^{-1}$ or $(x^*)^{-1}\delta_p(x)x^*$ in $w$.
So in this case, these three instances cannot be removed, and since two of them are both to a positive power or both to a negative power, there would be at least $4$ instances of $\sigma_p(x)$ in $w_0$, which is impossible.
So at most one element of $\supp w_0$ maps to an element with a given $x$ in its support under $\delta_p$.

From this we deduce that the matrix $(\delta_p)_*$ has diagonal entries of $1$, has off-diagonal entries of either $\pm1$ or $0$, and has only trivial entries away from the rectangular block sending elements dominated by $C_p$ to the image of $C_p\cap \supp w_0$.
Further, each row has at most one nonzero off-diagonal entry.
Then the block sending the image of $C_p\cap\supp w_0$ to itself must be invertible; all together these conditions indicate that there is a re-indexing of the basis that makes $(\delta_p)_*$ upper-triangular. 
An upper-triangular matrix where each row has at most one nonzero off-diagonal entry can be column reduced using each row operation at most once.
By Theorem~\ref{th:threeparts}, we have factored $\delta_p$ as a product of short-range transvections with distinct multipliers in $C_p$.

Then one can easily see that $\beta=\beta''\sigma_p(\delta\delta_p)\iota_p$ is a factorization satisfying the inductive hypothesis for the next step.
The lemma follows.
\end{proof}

Finally, we proceed to reducing peaks.

\begin{lemma}\label{le:prone}
Let $\tau_{x,y}\in\Omega_s$ with $x,y\in \supp w_0$.
Let $\alpha\in\langle\Omega^0_\ell\rangle$.
Suppose $\beta\in\langle\Omega_s\rangle$ is a product of transvections of the form $\tau_{a,b}$ for various $a\in \supp w_0$, $\pg{a}\neq\pg{x}$, such that $x$ does not strictly dominate $a$.
Further suppose that $\alpha\tau_{x,y}\beta([w_0])$ has the same length and support as $w_0$.

Then we can find $(B,x)\in \Omega$, and $\alpha',\alpha''\in\langle\Omega^0_\ell\rangle$ such that $\alpha\tau_{x,y}\beta=\alpha'(B,x)\alpha''\beta$,
and $\alpha''\beta([w_0])$ and $(B,x)\alpha''\beta([w_0])$ have the same length and support as $w_0$.
\end{lemma}

\begin{proof}
By Theorem~\ref{th:threeparts}, we peak-reduce $\alpha$ with respect to $\tau_{x,y}\beta([w_0])$.
Then since $[w_0]$ is of minimal length in its $\AAG$--orbit by Corollary~\ref{co:lenmin}, 
we have 
$\alpha =\gamma_1\cdots\gamma_p\delta_1\cdots\delta_q$,
for some $\gamma_1,\ldots,\gamma_p,\delta_1,\ldots,\delta_q\in\Omega_\ell$,
where each $\gamma_i$ leaves the length of the word the same, and each $\delta_i$ shortens the word.
More precisely, if $|\tau_{x,y}\beta([w_0])|=|w_0|$, then $q=0$;
if $q>0$ then for each $i$ we have $|\delta_i\cdots\delta_q\tau_{x,y}\beta([w_0])|<|\delta_{i+1}\cdots\delta_q\tau_{x,y}\beta([w_0])|$;
and if $p>0$ then for each $i$ we have $|\gamma_i\cdots\gamma_p\delta_1\cdots\delta_q\tau_{x,y}\beta([w_0])|=|\gamma_{i+1}\cdots\gamma_p\delta_1\cdots\delta_q\tau_{x,y}\beta([w_0])|$.

Since $\beta$ is a product of transvections with multipliers in $\supp w_0$, we know that $\supp\tau_{x,y}\beta([w_0])\subset\supp w_0$.
By Lemma~\ref{le:suppmin}, we know that they are equal.
Since each $\delta_i$ decreases length, we know that:
\[\supp\delta_{i}\cdots\delta_q\tau_{x,y}\beta([w_0])\subset \supp\delta_{i+1}\cdots\delta_q\tau_{x,y}\beta([w_0])\]
Again from Lemma~\ref{le:suppmin},  we know $\supp\delta_{i}\cdots\delta_q\tau_{x,y}\beta([w_0])=\supp w_0$ for each $i$.

Temporarily fix $i$.
The automorphism $\delta_i=(A,a)$ for some $a\in L$.
Since $\delta_i$ decreases length, we know that $\pg{a}\in\supp w_0$.
By Lemma~\ref{le:suppmin}, we know $\delta_i$ cannot remove all the instances of $a^{\pm1}$ from $\delta_{i+1}\cdots\delta_q\tau_{x,y}\beta([w_0])$, so there must be extra instances of $a$ in $\delta_{i+1}\cdots\delta_q\tau_{x,y}\beta([w_0])$.
These extra instances must have been put there by $\tau_{x,y}\beta$ (since the other $\delta_j$ automorphisms are length-decreasing), so we deduce that either $\pg{a}=\pg{x}$, or that there is some $z\in\supp w_0$ such that $a\in\supp \beta(z)$.
By the hypotheses on $\beta$, this tells us that if $\pg{a}\neq\pg{x}$, then $x$ does not strictly dominate $a$.

Now we consider what happens when we try to move $\tau_{x,y}$ to the left across $\delta_i=(A,a)$.
From Lemma~\ref{le:validsubs} and Definition~\ref{de:sort}, both of Day~\cite{fpraag}, we know that they commute (at least in $\OAG$) unless $\pg{a}=\pg{y}$.
Without loss of generality we temporarily assume $a=y$.
In this case, conjugating $\tau_{x,y}$ across $\delta_i$ introduces a short range element $s((A-a+x,x))$ and a long-range element $\ell((A-a+x,x))$.
However, since $x$ does not strictly dominate $a$ and $a=y$, we know $x\geq a$ and therefore $x\sim a$.
If $x\sim a$, then the element $s((A-a+x,x))=1$.
So in any case, we add at most a single new long-range element (working in $\OAG$) and no new short-range elements.
In returning to $\AAG$ it is possible that we introduce an inner automorphism, which is a product of long-range automorphisms.
So we have shown that there is an element $\phi_i\in\langle\Omega^0_\ell\rangle$ such that
$\delta_i\tau_{x,y}=\tau_{x,y}\phi_i$.

We rewrite $\delta_1\cdots\delta_p\tau_{x,y}$ as $\tau_{x,y}\phi_1\cdots\phi_p$.
If 
$$|\tau_{x,y}\phi_1\cdots\phi_p\beta([w_0])|=|\phi_1\cdots\phi_p\beta([w_0])|$$
 then we are done; if we set $(B,x)=\tau_{x,y}$, set $\alpha'=\gamma_1\cdots\gamma_q$ and set $\alpha''=\phi_1\cdots\phi_p$ then the conclusions hold.
So assume $\tau_{x,y}$ decreases the length of the word.
From the setup, we know that $\tau_{x,y}\phi_1\cdots\phi_p\beta([w_0])$ has the same length and support as $w_0$.
Then we know that $y$ and $y^{-1}$ both appear only once in $\phi_1\cdots\phi_p\beta([w_0])$.
This means that $\tau_{x,y}$ decreases the length by $2$, removing an instance of $x$ and $x^{-1}$ each.
By the form of $\beta$, we know that $\beta([w_0])$ only has a single $x$ and a single $x^{-1}$, so $\phi_1\cdots\phi_p$ must increase the number of instances of $x$.
We have a word $\psi_1\cdots\psi_r$ in $\Omega^0_\ell$ that is a peak-reduced factorization of $\phi_1\cdots\phi_p$ with respect to $\beta([w_0])$.
Some automorphism $\psi_i$ adds an extra instance of $x$ and in doing so increases the length by $2$.
Since the factorization is peak-reduced, this automorphism must be $\psi_1=(B',x)$ (without loss of generality, we assume the multiplier is $x$ and not $x^{-1}$, since $\tau_{x,y}=\tau_{x^{-1},y^{-1}}$).
We set $B=B'\cup\{y\}$ to get $\tau_{x,y}(B',x)=(B,x)\in \Omega$.
Since $|\psi_1\cdots\psi_r\beta([w_0])|=|w_0|+2$, we know that $|\psi_2\cdots\psi_r\beta([w_0])|=|w_0|$, and therefore also that $|(B,x)\psi_2\cdots\psi_r\beta([w_0])|=|w_0|$.
Then by setting $\alpha'=\gamma_1\cdots\gamma_p$ and $\alpha'=\psi_2\cdots\psi_r$, we are done.
\end{proof}

\begin{theorem}\label{th:prw0}
If $\gamma\in\AAGo$ with $\gamma([w_0])=[w_0]$, then there is a factorization of $\gamma$ as a product of elements of $\Omega$ that is peak-reduced with respect to $[w_0]$.
\end{theorem}

\begin{proof}
By Theorem~\ref{th:threeparts}, we factor $\gamma$ as $\alpha\beta$, where $\beta\in\langle\Omega_s\rangle$ and $\alpha\in\langle\Omega^0_\ell\rangle$.
By Lemma~\ref{le:specialsrform}, we write $\beta$ as $\sigma'\tau_{x_1,y_1}\cdots\tau_{x_r,y_r}\iota$, where $\sigma'$ is a permutation, $\iota$ is $w_0$--irrelevant, and the $\tau_{x_i,y_i}$ are short-range transvections such that $\{x_1,\ldots, x_r\}$ lie over distinct vertices and each $x_i$ is domination-minimal among $\{x_i,\ldots,x_r\}$.
Now we rewrite $\alpha\sigma'$ as $\sigma'\alpha'$, where $\alpha'\in\langle\Omega^0_\ell\rangle$.
By Theorem~\ref{th:threeparts}, we have a factorization $\alpha'=\sigma''\delta_1\cdots\delta_p$ which is peak-reduced with respect to $\tau_{x_1,y_1}\cdots\tau_{x_r,y_r}\iota([w_o])$, where $\sigma''$ is a permutation automorphism and each $\delta_i$ is a non-permutation automorphism in $\Omega^0_\ell$.
We set $\sigma=\sigma'\sigma''$, so that we have
$\gamma=\sigma\delta_1\cdots\delta_p\tau_{x_1,y_1}\cdots\tau_{x_r,y_r}\iota$.

Since each $x_i\in\supp w_0$, we deduce that no $\tau_{x_i,y_i}$ changes the support of $w_0$ (if it did, this would contradict Lemma~\ref{le:suppmin}).
By Corollary~\ref{co:lenmin}, we know that $|\tau_{x_1,y_1}\cdots\tau_{x_r,y_r}([w_0])|\geq|w_0|$, so since $\alpha$ is peak-reduced, each $\delta_i$ either shortens $\delta_{i+1}\cdots\delta_p\tau_{x_1,y_1}\cdots\tau_{x_r,y_r}([w_0])$ or leaves its length unchanged.
Since $\delta_i$ acts by a single multiplier, this means $\delta_i$ either leaves $\supp\delta_{i+1}\cdots\delta_p\tau_{x_1,y_1}\cdots\tau_{x_r,y_r}([w_0])$ the same or removes a single element.
However, if this support is equal to $\supp w_0$, removing an element would contradict Lemma~\ref{le:suppmin}.
So inductively, we deduce that:
$$\supp\delta_{1}\cdots\delta_p\tau_{x_1,y_1}\cdots\tau_{x_r,y_r}([w_0])=\supp w_0$$
Since $\delta_{1}\cdots\delta_p\tau_{x_1,y_1}\cdots\tau_{x_r,y_r}([w_0])$ differs from $[w_0]$ by the permutation $\sigma$, we know that their lengths are the same.

Now inductively assume that we have written $\delta_1\cdots\delta_p\tau_{x_1,y_1}\cdots\tau_{x_r,y_r}$ as a product
$\phi_0(A_1,x_1)\phi_1\cdots\phi_{j-2}(A_{j-1},x_{j-1})\phi'_{j-1}\tau_{x_j,y_j}\cdots\tau_{x_r,y_r}$,
with $\phi'_{j-1}\in\langle\Omega^0_\ell\rangle$, and with $\phi_i\in\langle\Omega^0_\ell\rangle$ and $(A_i,x_i)\in\Omega$ for each $i$.
Also suppose that for each $i$,  
\begin{align*}
&(A_i,x_i)\phi_i\cdots(A_{j-1},x_{j-1})\phi'_{j-1}\tau_{x_j,y_j}\cdots\tau_{x_r,y_r}([w_0])\\
\tag*{\mbox{and}}&\phi_i\cdots(A_{j-1},x_{j-1})\phi'_{j-1}\tau_{x_j,y_j}\cdots\tau_{x_r,y_r}([w_0])
\end{align*}
have the same length and support as $w_0$.
The base case for this induction has $\phi'_0=\delta_1\cdots\delta_p$.

Then we simply apply Lemma~\ref{le:prone} to $\phi'_{j-1}\tau_{x_j,y_j}\cdots\tau_{x_r,y_r}$ and get the same statement with $j$ instead of $j-1$.
After we have done this $r$ times, we get
\[\gamma=\sigma\phi_0(A_1,x_1)\phi_1\cdots(A_r,x_r)\phi'_r\iota\]
Peak-reduce each $\phi_i$ with respect to $(A_{i},x_{i})\phi_{i}\cdots(A_r,x_r)\phi_r([w_0])$, peak-reduce $\phi'_r$ with respect to $[w_0]$, and 
write out $\iota$ as a product of $w_0$--irrelevant transvections; this is a peak-reduced factorization of $\gamma$.
\end{proof}

The following ideas appear for free groups in Lyndon--Schupp~\cite{ls} and are closely related to the work of McCool in~\cite{mccool}.

\begin{definition}\label{de:whiteheadg}
We construct a labeled, directed multi-graph $\overline \Delta$ whose vertices are conjugacy classes of $A_\Gamma$ with length equal to $|w_0|=4k$, where we place a directed edge from $[u]$ to $[v]$ if there is a Whitehead automorphism $\alpha\in\Omega$ with $\alpha([u])=[v]$.
We label this directed edge by $\alpha$.
The \emph{Whitehead graph} $\Delta$ of $[w_0]$ is the (undirected) connected component of $[w_0]$ in $\overline \Delta$. 
\end{definition}

Since there are only finitely many words of length $4k$, there are only finitely many conjugacy classes of length $4k$.
Since $\Omega$ is finite, this means that $\Delta$ is a finite graph with only finitely many edges between any two vertices.

\begin{corollary}\label{co:w0gpfg}
The group $(\AAGo)_{[w_0]}$ of automorphisms in $\AAGo$ preserving $[w_0]$ is finitely generated. 
\end{corollary}

\begin{proof}
A path in $\Delta$ determines an element of $\AAGo$ by composing the labels along the edges.
Further, if $\alpha$ is the automorphism determined by a path from the vertex $[w_1]$ to the vertex $[w_2]$, we know that $\alpha([w_1])=[w_2]$ (this is true by definition for a path of length one and remains true under concatenation).
In particular, this defines a map $\pi_1(\Delta,[w_0])\to (\AAGo)_{[w_0]}$.
If $\alpha\in(\AAGo)_{[w_0]}$, then by Theorem~\ref{th:prw0}, there is a factorization $\beta_k\cdots\beta_1$ of $\alpha$ by elements of $\Omega$ that is peak reduced with respect to $[w_0]$.
By Corollary~\ref{co:lenmin}, this means that for each $i$, $\beta_i\cdots\beta_1([w_0])$ is a vertex in $\Delta$, and $\beta_i$ is an edge from $\beta_{i-1}\cdots\beta_1([w_0])$ to $\beta_i\cdots\beta_i([w_0])$.
So $\beta_k\cdots\beta_1$ describes a path in $\Delta$ that maps to $\alpha$.
So the finitely generated group $\pi_1(\Delta,[w_0])$ surjects on $(\AAGo)_{[w_0]}$.
\end{proof}

\begin{remark}
There are normal forms for elements of $A_\Gamma$ (see, for example, van Wyk~\cite{vw}), so there is an effective procedure to produce the Whitehead graph of $w_0$.
Of course, this means that there is a procedure to give a generating set for $(\AAGo)_{[w_0]}$.
Unfortunately, $\Delta$ can be large for simple examples and it appears to be difficult to use this method to write down specific generating sets.
McCool has explored this for the case where $A_\Gamma$ is a free group and $(\AAGo)_{[w_0]}$ is a mapping class group in~\cite{mcmcg}.
This procedure does not produce the familiar generating sets for the mapping class group given by Dehn twists.
\end{remark}

\subsection{Finite generation of $\Mod(\Gamma,w_0,Q)$}
\newcommand\trans{\mathrm{trans}\,}

This subsection is devoted to finishing the proof of Theorem~\ref{mt:finitelygenerated}.
Recall the Whitehead graph $\Delta$ from Definition~\ref{de:whiteheadg}.   

\begin{lemma}\label{le:whtreeform}
The graph $\Delta$ has a maximal tree $T$ such that the set of all edges in $T$ that are non-permutation automorphisms forms a subtree containing $[w_0]$.
\end{lemma}

\begin{proof}
We take $\Delta'$ to be the connected component of $[w_0]$ in the subgraph of $\Delta$ gotten by deleting the edges labeled with permutation automorphisms.
We take $T'$ to be a maximal tree for $\Delta'$.

Now, for each vertex $[w]$ of $\Delta$ not in $\Delta'$, there is a path $p$ from $[w_0]$ to $[w]$ in $\Delta$ (paths in $\Delta$ are written in function composition order).
If $\sigma$ is a permutation automorphism in $\Omega$, and $\alpha$ is a non-permutation Whitehead automorphism in $\Omega$ such that $\alpha\cdot \sigma$ is a segment in $p$, then by Equation~\ref{eq:R6} of Day~\cite{fpraag}, $\sigma \cdot (\sigma^{-1}\alpha\sigma)$ is another segment of length two in $\Delta$ connecting the same initial and terminal vertices.
Note that since $\sigma$ leaves $\supp w_0$ invariant, we know that $\sigma^{-1}\alpha\sigma\in\Omega$.
We modify $p$ by substituting this second segment in for the first one.
By repeating substitutions like this whenever possible, and multiplying the permutation automorphisms together as a single permutation, we get a path $p'$ from $[w_0]$ to $[w]$ of the form
\[\sigma_{[w]}\cdot (A_m,a_m) \cdots (A_1,a_1)\]
where $\sigma_{[w]}$ is a permutation automorphism and each $(A_i,a_i)\in\Omega$.

We already have a path in $T'$ from $[w_0]$ to $(A_m,a_m)\cdots(A_1,a_1)([w_0])$, so we can add the edge $\sigma_{[w]}$ starting at $(A_m,a_m)\cdots(A_1,a_1)([w_0])$ to $T'$ to get a tree containing $[w]$.
It is obvious that if we add an edge gotten in this manner to $T'$ for each vertex not in $T'$, we will get a maximal tree for $\Delta$ satisfying the conclusions of the lemma.
\end{proof}

\begin{definition}
For $(A,a)\in\Omega$, the \emph{transvection set} $\trans (A,a)$ is the set of $x\in X$ with $x\in A$ or $x^{-1}\in A$, but not both.
\end{definition}

\begin{lemma}\label{le:whedgeform}
The graph $\Delta$ of $w_0$ has a maximal tree $T$ satisfying the following condition:
for each edge $\alpha$ originating at a vertex $[w]$, either $\alpha$ is a permutation automorphism or a Whitehead automorphism $\alpha=(A,a)$ with $a\in\supp w$ and $\trans \alpha\subset\supp w$. 
\end{lemma}

\begin{proof}
Start with an arbitrary maximal tree $T_0$.
Suppose we have an edge $(A,a)$ of $T_0$ originating at a vertex $[w]$.
Fix a cyclic representative $w$ of $[w]$, and consider the obvious representative $w'$ of $(A,a)([w])$ based on $w$.
If $a\notin \supp w$, then $w'$ is the same as $w$ with some instances of $a$ added in.
Since these are both graphically reduced representatives of conjugacy classes of the same length, we deduce that in fact, $(A,a)([w])=[w]$.
However, since $T_0$ is a tree, we cannot have a loop $(A,a)$, so it must be that $a\in\supp w$.
If $\trans (A,a)\not\subset \supp w$, then we can rewrite $(A,a)$ as $(A_1,a)(A_2,a)$ where $\trans (A_1,a)\subset\supp w$ and $(A_2,a)$ is a product of transvections with $\trans (A_2,a)\cap\supp w=\emptyset$. 
In this case, we know that $(A_2,a)([w])=[w]$, and therefore $(A_1,a)([w])=(A,a)([w])$.
We replace the edge $(A,a)$ with the edge $(A_1,a)$.
Of course, we can repeat this procedure with each edge of $T_0$ to obtain a tree $T$ which satisfies the conclusions of the lemma.
\end{proof}
Note that if the tree $T_0$ above satisfies the conclusions of Lemma~\ref{le:whtreeform}, then the final tree $T$ does as well.
So at this point we fix a maximal tree $T$ in $\Delta$ that satisfies the conclusions of Lemma~\ref{le:whtreeform} and Lemma~\ref{le:whedgeform}.
Let $T'$ be the subtree of $T$ whose edges are non-permutation Whitehead automorphisms.
For each vertex $[w]\in T'$, let $\alpha_{[w]}\in\AAGo$ be the product of edge labels in the edge path in $T'$ from $[w_0]$ to $[w]$.

\begin{lemma}\label{le:awform}
For $[w]\in T'$, we have $\supp [w]=\supp w_0$.
In particular, $\alpha_{[w]}$ is an automorphism of the form $(A_m,a_m)\cdots(A_1,a_1)$ where for each $i$, $a_i\in \supp w_0$ and $\trans(A_i,a_i)\subset\supp w_0$.
\end{lemma}

\begin{proof}
Suppose $[w],[w']\in \Delta$ and we have $(A,a)$ with $a\in\supp [w]$ and $(A,a)([w])=[w']$.
Then $\supp [w']\subset \supp[w]\cup\{a\}=\supp[w]$.
Since the vertices of $\Delta$ are conjugacy classes of the same length as $[w_0]$, it follows that $\supp [w']=\supp[w]$.
It then follows from the definition of $T'$ that for all $[w]\in T'$, we have $\supp [w]=\supp w_0$.
The second statement in the lemma then follows from the first one.
\end{proof}

\begin{definition}
For $[w]\in T'$ (possibly $[w]=[w_0]$), the automorphism $\alpha_{\beta([w])}^{-1}\beta\alpha_{[w]}$ in $(\AAGo)_{[w_0]}$ is an \emph{edge generator} if $\beta$ is an edge in $\delta$ originating at $[w]$ with $\beta([w])\in T'$ and either 
\begin{itemize}
\item $\beta$ is a permutation automorphism fixing $(\supp Q)^{\pm1}$ pointwise, or
\item $\beta=(B,b)$ and $\trans\beta\subset \supp w_0$.
\end{itemize}
Define the set $S_e$ to be the set of edge generators.

Define the set $S_i\subset (\AAGo)_{[w_0]}$, the set of \emph{$w_0$--independent generators}, to be the set of elements $\tau_{a,b}$ where $b\in\supp Q$ and $a\in X$ with $a\geq b$, together with the inversions with respect to elements of $\supp Q$.

Define the set $S_Q\subset (\AAGo)_{([w_0],Q)}$, the set of \emph{lifted $Q$--transvections and $Q$--inversions}, to be the set of permutation automorphisms inducing a $Q$--inversion together with those products of transvections of length 1 or 2 that induce standard dominated $Q$--transvections in $\Aut H_\Gamma$.

Define the set $S_k\subset (\AAGo)_{([w_0],Q)}$, the set of \emph{kernel generators}, to be the set of elements of the following forms:
\begin{itemize}
\item automorphisms $\tau_{[x,y],c}$ (as in Section~\ref{se:KZisker}) where $x,y\in X$, $c\in\supp Q$, and $x,y\geq c$,
\item partial conjugations $c_{x,\{c\}}$, where $x\in X$, $c\in\supp Q$, and $x\geq c$, and
\item conjugations $c_x$, where $x\in X$.
\end{itemize}
\end{definition}
We call $S_k$ the set of kernel generators because these generators will be part of our generating set and they lie in the kernel of the homology representation.

Our next intermediate goal is the following.
\begin{proposition}\label{pr:Sektgens}
The finite set $S_e\cup S_k\cup S_i$ generates $(\AAGo)_{[w_0]}$.
\end{proposition}
We will prove some lemmas before proving this proposition.

\begin{lemma}\label{le:kerrewrite}
Let $b\in(\supp Q)^{\pm1}$.
Suppose $\alpha=(A_m,a_m)\cdots(A_1,a_1)$ is a product of Whitehead automorphisms such that $\pg{a_i}\neq b$ and $\trans(A_i,a_i)\subset \supp w_0$ for each $i$.
Further suppose we have $\beta_1,\beta_2\in\AAGo$ both of which which satisfy the following three conditions:
(1) $\beta_i$ fixes each $x\in X-\{b\}$; (2) $\beta_i(b)$ contains only a single instance of $b$ and no instance of $b^{-1}$; and (3) for each $y\in\supp\beta_i(b)$, we have $y\geq b$ or $y=b$.
Finally, suppose that $(\alpha^{-1}\beta_1\alpha\beta_2)_*\in \Aut H_\Gamma$ is the identity.

Then we have $\alpha^{-1}\beta_1\alpha\beta_2 \in\langle S_k\rangle$.
\end{lemma}

\begin{proof}
Let $\gamma=\alpha^{-1}\beta_1\alpha\beta_2$.
If $x\in X-\{b\}$, then since each $\pg{a_i}\neq b$, we know that $b\notin\supp\alpha([x])$ and therefore that $\gamma$ fixes $x$.
Since $b\notin\trans(A_i,a_i)$ for any $i$, 
if we alter $\alpha$ by an inner automorphism, we may assume that each $(A_i,a_i)$ fixes $b$.
Since we aim to show $\gamma$ is in $\langle S_k\rangle$, which contains the inner automorphisms, we can do this.
If some $(A_i,a_i)$ conjugates some $x\geq b$ while fixing $b$, we can deduce that $a_i\geq b$, and it follows that each element of $\supp \gamma(b)$ is either equal to $b$ or dominates $b$.
Since each $a_i\neq b$, we know that only a single instance of $b$ appears in $\gamma(b)$.
Also, we know that $\gamma$ fixes the image of $b$ in $H_\Gamma$, so each element of $X-\{b\}$ that appears in $\gamma(b)$ appears in pairs of opposite exponent.

We claim that we can reduce $\gamma$ to the identity by a series of applications of elements of $S_k$.
Let the cyclic word $v_0$ be a graphically reduced representative of $\gamma(b)$; by the previous reasoning, we know that $v_0$ contains a single instance of $b$.

Suppose the $b$ in $v_0$ is in a subsegment $xby$ for $x,y\in L$.
Note that $v_0$ with $bxy$ substituted for $xby$ represents $c_{x,\{b\}}([v_0])$ and that $v_0$ with $xyb$ substituted for $xby$ represents $c_{y,\{b\}}^{-1}([v_0])$.
In this manner, by applying some partial conjugations from $S_k$, we can send $[v_0]$ to a conjugacy class represented by $v_0$ with $b$ moved to any position in the cyclic word.
If the $b$ in $v_0$ is in a subsegment $bxy$ with $\pg{x}\neq\pg{y}$, then $v_0$ with $byx$ substituted for $bxy$ is a representative for $\tau_{[y,x],b}([v_0])$.
So by applying some elements from $S_k$, we can send $[v_0]$ to a conjugacy class represented by $v_0$ with the two letters to the right of $b$ swapped.

So, to shorten $v_0$, identify an instance of some $x$ and an instance of $x^{-1}$ in $v_0$, apply elements of $S_k$ to move $b$ to the left of $x$, apply an element to move $x$ to the right (by a swap), move $b$ to the right, and repeat, until $x$ is next to $x^{-1}$ and they cancel.
Note that all of these moves are allowed since the elements appearing in $v_0$ other than $b$ dominate $b$, and also note that these moves fix every element other than $b$.
By this procedure we can shorten $v_0$ until only $b$ remains, and we have produced an automorphism $\delta\in\langle S_k\rangle$ with $\delta=\gamma^{-1}$.
\end{proof}

\begin{lemma}\label{le:rewritegens}
Let $[w]\in T'$ and let $b\in\supp Q$.
For any $a\in X$ with $a\geq b$, we have an automorphism $\gamma$ that is a product of transvections acting only on $b$, such that:
\[\alpha_{[w]}^{-1}\tau_{a,b}\alpha_{[w]}\gamma\in\langle S_k \rangle\]

For any element $\alpha=\alpha_{(C,c)([w])}^{-1}(C,c)\alpha_{[w]}\in S_e$, with $\pg{c}\neq b$, 
 there is a product $\gamma$ of transvections acting only on $b$ such that:
\[\alpha^{-1}\tau_{a,b}\alpha\gamma^{-1}\in \langle S_k\rangle\]

If $\beta\in \langle S_k\rangle$ fixes every element of $X-\{b\}$, then:
\[\alpha_{[w]}^{-1}\beta\alpha_{[w]} \in \langle S_k\rangle\]
\end{lemma}

\begin{proof}
For the first statement, note that 
\[(\alpha_{[w]})^{-1}_*[a]=\sum_{i=1}^m p_i[c_i]\in H_\Gamma\]
for some $c_i\in X$ and nonzero integers $p_i$.
Then for each $i$ we have $c_i\geq a$ by Lemma~\ref{le:knowdom}, so $c_i\geq b$ and
we can take $\gamma=\tau_{c_1,b}^{p_1}\cdots\tau_{c_m,b}^{p_1}$.

Similarly, it follows from Lemma~\ref{le:knowdom} that the element $\gamma$ needed for the second statement also exists.

Then the lemma is immediate from Lemma~\ref{le:awform} and Lemma~\ref{le:kerrewrite}.
\end{proof}

\begin{lemma}\label{le:morerewrite}
Let $[w]\in T'$ and $\gamma\in\langle S_i\rangle$.
Then there is an element $\gamma'\in\langle S_i\rangle$ such that 
$\alpha_{[w]}^{-1}\gamma\alpha_{[w]}\gamma'$
is in $\langle S_k\rangle$.
\end{lemma}

\begin{proof}
We proceed by induction on the $S_i$--length of $\gamma$.
Suppose $\gamma=\gamma_0\beta$ where $\beta\in S_i$ and we have some $\gamma'_0\in\langle S_i\rangle$ such that $\delta= \alpha_{[w]}^{-1}\gamma_0\alpha_{[w]}\gamma'_0\in \langle S_k\rangle$.
If $\beta$ is the inversion with respect to any element of $\supp Q$, then a computation shows that $\alpha_{[w]}$ commutes with $\beta$, and therefore
\[\alpha_{[w]}^{-1}\gamma\alpha_{[w]}\beta^{-1}\gamma'_0=\delta\in\langle S_k\rangle\]
and we can take $\gamma'=\beta^{-1}\gamma'_0$.

If $\beta$ is a transvection $\tau_{a,b}$ with $b\in \supp Q$, then by Lemma~\ref{le:rewritegens} we have a $\gamma''\in\langle S_i\rangle$ with $\delta'=\alpha_{[w]}^{-1}\beta\alpha_{[w]}\gamma''\in\langle S_k\rangle$.
Set $\delta''=(\gamma'_0)^{-1}\delta'\gamma'_0$.
From Lemma~\ref{le:KZnormal} (with $Z=\supp Q$), we have $\delta''\in\langle S_k\rangle$.
We set $\gamma'=\gamma''\gamma'_0$, and we have
\begin{equation*}\begin{split}\alpha_{[w]}^{-1}\gamma_0\beta\alpha_{[w]}\gamma''\gamma'_0&=\alpha_{[w]}^{-1}\gamma_0\alpha_{[w]}\delta'\gamma'_0\\
&=\delta(\gamma'_0)^{-1}\delta'\gamma'_0=\delta\delta''
\end{split}\end{equation*}
which is in $\langle S_k\rangle$.
\end{proof}

\begin{proof}[Proof of Proposition~\ref{pr:Sektgens}]
First note that $S_e$ is finite because $\Delta$ is finite, and $S_i$ and $S_k$ are finite because $X$ is finite.
Now suppose we have an edge $\alpha$ between two vertices $[w_1]$ and $[w_2]$ of $\Delta$.
We know that $[w_i]=\sigma_i([v_i])$ where $\sigma_i$ is a possibly trivial permutation automorphism and $[v_i]\in T'$, for $i=1,2$.
Define the set $S\subset \AAG$ to be the set of elements of the form $\alpha_{[v_2]}^{-1}\sigma_2^{-1}\alpha\sigma_1\alpha_{[v_1]}$, indexed over all edges $\alpha$ of $\Delta$.
Since $T'$ is a maximal tree for $\Delta$, the elements of $S$ describe a generating set for $\pi_1(\Delta,[w_0])$, and as explained in Corollary~\ref{co:w0gpfg}, they therefore generate $(\AAGo)_{[w_0]}$.

Since it is obvious that $S_e\cup S_k\cup S_i\subset (\AAGo)_{[w_0]}$, we prove the lemma by showing that $S\subset\langle S_e\cup S_k\cup S_i\rangle$.
Consider an arbitrary element of $S$:
\[\beta=\alpha_{[v_2]}^{-1}\sigma_2^{-1}\alpha\sigma_1\alpha_{[v_1]}\]
If $\alpha$ is a permutation automorphism, then we write $\sigma_2^{-1}\alpha\sigma_1$ as a single permutation $\sigma_3$. 
Since $[v_1]$ and $[v_2]$ are both in $T'$, we know that $\supp [v_1]=\supp [v_2]=\supp w_0$, and therefore $\sigma_3$ leaves $\supp w_0$ invariant. 
It follows from this and the fact that $\sigma_3$ is in $\AAGo$ that $\sigma_3$ factors as a product of a permutation automorphism $\sigma_{w_0}\in \AAGo$ that fixes $(\supp Q)^{\pm1}$ pointwise and a permutation automorphism $\sigma_Q\in\AAGo$ that fixes $(\supp w_0)^{\pm1}$ pointwise.
Since $\sigma_Q$ is in $\AAGo$ and fixes $(\supp w_0)^{\pm1}$, it follows that $\sigma_Q\in\langle S_i\rangle$.
We know $\supp [v_1]=\supp w_0$, so $\sigma_Q$ fixes $[v_1]$ and therefore $\sigma_{w_0}([v_1])=[v_2]$ and $\alpha_{[v_2]}^{-1}\sigma_{w_0}\alpha_{[v_1]}\in S_e$.
Then $\beta$ will be in $\langle S_e\cup S_k\rangle$ if $\alpha_{[v_1]}^{-1}\sigma_Q\alpha_{[v_1]}$ is.
Since $\sigma_Q\in \langle S_i\rangle$, Lemma~\ref{le:morerewrite} says that there is a $\gamma\in\langle S_i\rangle$ with $\alpha_{[v_1]}^{-1}\sigma_Q\alpha_{[v_1]}\gamma\in\langle S_k\rangle$.
The proposition follows in this case.

So assume that $\alpha=(A,a)$.
By replacing $\alpha$ with $\sigma_2^{-1}\alpha\sigma_2$ and $\sigma_1$ with $\sigma_2^{-1}\sigma_1$, we may assume that $\sigma_2=1$.
If we set $[w_3]=\alpha^{-1}\alpha_{[v_2]}([w_0])$, then it follows from the construction of $T'$ that $[w_3]\in T'$.
Then we know that $\alpha_{[w_3]}^{-1}\sigma_1\alpha_{[v_1]}$ is in $S_e$, and therefore $\beta$ is in $\langle S_e\cup S_k\cup S_i\rangle$ only if the element
\[\beta'=\alpha_{[v_2]}^{-1}\alpha\alpha_{[w_3]}\]
is as well.
We may rewrite $\alpha=(A',a)\gamma$, where $\gamma\in\langle S_i\rangle$ and $\trans(A',a)\subset \supp w_0$. 
Since $[w_3]\in T'$, we know that $\supp [w_3]=\supp [w_0]$, and $\gamma$ fixes $[w_3]$.
We may rewrite $\beta'$ as the product of the element
$\alpha_{[v_2]}^{-1}(A',a)\alpha_{[w_3]}$, which is in $S_e$, with the element
$\alpha_{[w_3]}^{-1}\gamma\alpha_{[w_3]}$.
By Lemma~\ref{le:morerewrite}, there is an element $\gamma'\in\langle S_i\rangle$ such that $\alpha_{[w_3]}^{-1}\gamma\alpha_{[w_3]}\gamma'$ is in $\langle S_k\rangle$.
The proposition follows.
\end{proof}

We proceed by showing that we can do better:
\begin{proposition}\label{pr:1ofeach}
Any element of $(\AAGo)_{[w_0]}$ can be written as the product of a single element of $\langle S_e\cup S_k\rangle$ and a single element of $\langle S_i\rangle$.
\end{proposition}

\begin{sublemma}\label{sl:substitute}
Suppose $(B,b)$ is a Whitehead automorphism, $a\in L$ with $a\geq b$ and $a,a^{-1}\notin B$.
Suppose $[u]$ is an element or conjugacy class in $A_\Gamma$ with $\pg{b}\notin \supp [u]$ and $v$ is a graphically reduced word or cyclic word representing $(B,b)([u])$.
Then if $v'$ is $v$ with all instances of $b$ replaced by $a$ and all instances of $b^{-1}$ replaced by $a^{-1}$, then $v'$ is a representative of $(B-b+a,a)([u])$.
\end{sublemma}
\begin{proof}
Note that since $a\geq b$, $(B-b+a,a)$ is well defined by Lemma~\ref{le:whdef} of Day~\cite{fpraag}.
Pick a graphically reduced representative $u$ for $[u]$; obtain a representative $\tilde v$ for $(B,b)([u])$ by applying $(B,b)$ letter-by-letter to $u$.
It is immediate that if $\tilde v'$ is $\tilde v$ with these substitutions, then $\tilde v$ represents $(B-b+a,a)$.
Note that $\pg{b}\notin\supp\tilde v$. 
Since $a$ commutes with every letter that $b$ commutes with (except possibly $b^{\pm1}$), each time we modify $\tilde v$ by swapping two adjacent, commuting letters, or by making a graphic reduction, we can make a parallel modification to $\tilde v'$ and still have representatives that differ by the described substitution and represent the same two elements.
Since we can get from any representative of $(B,b)([u])$ to the representative $v$ by such moves, we have proven the statement.
\end{proof}

\begin{lemma}\label{le:nobadcase}
Suppose $\beta=(B,b)\in\Omega$ such that $\pg{b}\notin\supp w_0$ and for some $[w]\in T'$, $\alpha_{[w]}^{-1}\beta\alpha_{[w]}\in S_e$.
If $a\in(\supp w_0)\cap(\trans\beta)$, then $a\not\sim b$.
\end{lemma}

\begin{proof}
Suppose for contradiction that $a\sim b$.
By the construction of $T'$, $w$ has the same length and support as $w_0$, so by Lemma~\ref{le:suppmin}, there is a single instance of $a$ and a single instance of $a^{-1}$ in $w$.
So write $w$ as the graphically reduced cyclic word $aua^{-1}v$.
Then $[u,a]\neq 1$ and $[v,a]\neq 1$.  
Since $a\sim b$, we know that $[u,b]\neq1$ and $[v,b]\neq 1$ as well.
Since $a\in\trans\beta$, we may assume that $a\in B$ and $a^{-1}\notin B$ (the case where $a^{-1}\in B$ and $a\notin B$ is similar).
We know $\beta([w])=[w]$ and $\beta(a)=ab$.
Since $[v,b]\neq 1$, $\beta$ cannot send $v$ to an element represented by a reduced word ending in $b^{-1}$ or beginning with $b$.
It follows that $\beta(u)=b^{-1}ub$ and $\beta(v)=v$.

Since $a\sim b$, we know from Lemma~\ref{le:whdef} of Day~\cite{fpraag} that $(B-b,a)$ is a well-defined Whitehead automorphism.
We know that $b$ does not appear in $u$ or $v$, so by Sublemma~\ref{sl:substitute} $(B-b,a)(u)=a^{-1}ua$ and $(B-b,a)(v)=v$.
Since $[u,a]\neq 1$, these expressions are graphically reduced.
Then $(B-b,a)(w)=aa^{-1}uaa^{-1}v=uv$, and $|uv|\leq|w|-2$.
This contradicts Corollary~\ref{co:lenmin}.
\end{proof}

\begin{lemma}\label{le:newmult}
Suppose $(B,b)\in\Omega$ such that $\pg{b}\notin\supp w_0$ and for some $[w]\in T'$, $\alpha_{[w]}^{-1}\beta\alpha_{[w]}\in S_e$.
Suppose $a\in L$ such that $a\geq b$ and $a,a^{-1}\notin B$.
Then $(B-b+a,a)$ fixes $[w]$, and $\alpha_{[w]}^{-1}(B-b+a,a)\alpha_{[w]}\in S_e$.
\end{lemma}

\begin{proof}
It is immediate from Sublemma~\ref{sl:substitute} that $(B-b+a,a)$ fixes $[w]$.
Note that $\trans(B-b+a,a)=\trans(B,b)\subset\supp w_0$, so $\alpha_{[w]}^{-1}(B-b+a,a)\alpha_{[w]}\in S_e$.
\end{proof}

\begin{lemma}\label{le:cinQrewrite}
Suppose $\alpha_{[w]}^{-1}(B,b)\alpha_{[w]}\in S_e$ and $\pg{b}\in \supp Q$.
If $\gamma$ is a product of transvections acting only on $b$ 
then $\gamma\alpha_{[w]}^{-1}(B,b)\alpha_{[w]}\gamma^{-1}$ is in $\langle S_e\cup S_k\rangle$.
\end{lemma}

\begin{proof}
Let $S_{[w]}$ be the union of $S_k$ with the set of $(C,c)\in\Omega$ such that $(C,c)([w])=[w]$ and $\trans(C,c)\subset\supp w_0$. 

As a base case, consider the effect of a single transvection $\tau_{a,b}$ on a $(B,b)$, where $a\in L$ and $a\geq b$.  
Note that if $a\in\trans(B,b)$, then $a\in\supp w_0$ and $b\geq a$.
Then $a\sim b$, contradicting Lemma~\ref{le:nobadcase}.
So we know $a\notin\trans(B,b)$.
Possibly by multiplying $(B,b)$ by an inner automorphism, we may assume $a\notin B$.
Then since $a\notin\trans(B,b)$, we know $a^{-1}\notin B$.
Since $\tau_{a,b}=(\{a,b\},a)$ and $b^{-1}\notin\{a,b\}$, we have the following special case of Equation~(\ref{eq:R4}) of Day~\cite{fpraag}:
\[(B,b)^{-1}\tau_{a,b}(B,b)=\tau_{a,b}(B-b+a,a)\]
We rephrase this as:
\[
\tau_{a,b}(B,b)\tau_{a,b}^{-1}=(B,b)(B-b+a,a)
\]
Similarly, note that:
\[
\tau_{a,b^{-1}}(B,b)\tau_{a,b^{-1}}^{-1}=(B-b+a^{-1},a^{-1})(B,b)
\]
By Lemma~\ref{le:newmult}, $(B-b+a,a)$ and $(B-b+a^{-1},a^{-1})$ both 
preserve $[w]$.
Of course, $\trans(B-b+a^{-1},a^{-1})=\trans(B-b+a,a)=\trans(B,b)\subset\supp w_0$.

Now consider $\tau_{a',b}$ for some $a'\in L$ with $a\geq b$.
If $\pg{a'}=\pg{a}$, then $\tau_{a,b}$ and $(B-b+a,a)$ commute.
Since $\trans(B-b+a,a)=\trans(B,b)$, and since by Lemma~\ref{le:nobadcase}, we know $a'\notin\trans(B,b)$, we know $a'\notin\trans(B-b+a,a)$.
Further, we know that $b, b^{-1}\notin B-b+a$.
Then by Equation~(\ref{eq:R3}) of Day~\cite{fpraag}, up to an inner automorphism, $\tau_{a',b}$ and $(B-b+a,a)$ commute.
Similarly, $\tau_{a',b^{-1}}$ and $(B-b+a,a)$ commute up to an inner automorphism.

So if $\gamma$ is a product of transvections acting only on $b$, then up to inner automorphisms, $\gamma(B,b)\gamma^{-1}$ is $(B,b)$ times some number of elements of the form $(B-b+a,a)$ for various $a\geq b$.
Then in particular, $\gamma(B,b)\gamma^{-1}$ is in $\langle S_{[w]}\rangle$

Now suppose $\gamma$ is a product of transvections acting only on $b$.
From Lemma~\ref{le:rewritegens}, there is a product $\gamma'$ of transvections acting on $b$ and a $\delta\in\langle S_e\rangle$ with $\gamma\alpha_{[w]}=\delta\alpha_{[w]}^{-1}\gamma'$.
We deduce from the previous paragraph there is an element $\beta\in\langle S_{[w]}\rangle$ with $\gamma'(B,b)\gamma'^{-1}=\beta$.
Then we have:
\[\gamma\alpha_{[w]}^{-1}(B,b)\alpha_{[w]}\gamma^{-1}=\delta\alpha_{[w]}^{-1}\beta\alpha_{[w]}\delta^{-1}\] 
By Lemma~\ref{le:rewritegens}, if $\delta'\in S_k$, then $\alpha_{[w]}^{-1}\delta'\alpha_{[w]}\in\langle S_k\rangle$.
And if $(C,c)\in S_{[w]}$, then $\alpha_{[w]}^{-1}(C,c)\alpha_{[w]}\in S_e$.
So $\alpha_{[w]}^{-1}\beta\alpha_{[w]}\in\langle S_e\cup S_k\rangle$, proving the lemma.
\end{proof}

\begin{lemma}\label{le:across1}
Suppose $b\in \supp Q$, $\gamma_1$ is a product of transvections acting on $b$ and $\alpha\in S_e\cup S_k$.
Then there is a $b'\in\supp Q$ and a product $\gamma_2$ of transvections acting on $b'$ such that:
\[\gamma_1\alpha\gamma_2^{-1}\in\langle S_e\cup S_k\rangle\]
\end{lemma}

\begin{proof}
First we note that if $\alpha\in S_k$, then we can take $\gamma_2=\gamma_1$ and the lemma follows from the identities in Sublemma~\ref{sl:conjidsconj} and Sublemma~\ref{sl:conjidstrans}.

Now suppose $\alpha=\alpha_{\beta([w])}^{-1}\beta\alpha_{[w]}$ and $\beta$ is a permutation automorphism.
Since $[w]$ and $\beta([w])$ are both in $T'$, we know that $\supp [w]=\supp\beta([w])=\supp w_0$.
So $\beta$ leaves $\supp Q$ invariant and we can set $b'=\pg{\beta^{-1}(b)}\in\supp Q$.
By Lemma~\ref{le:knowdom}, we can find a product $\gamma_2$ of transvections acting only on $b'$, such that $\gamma_2$ sends the image of $b'$ in $H_\Gamma$ to the same element that $\alpha_{[w]}^{-1}\beta^{-1}\alpha_{\beta[w]}\gamma_1^{-1}$ sends it to.
Let $\alpha'=\beta\alpha_{[w]}\beta^{-1}$ and let $\gamma'=\beta\gamma_2\beta^{-1}$.
Then by Equation~(\ref{eq:R6}) of Day~\cite{fpraag}, $\alpha'$ can be written as a product of non-permutation automorphisms with multipliers not equal to $\pg{b}$ and $\gamma'$ can be written as a product of transvections acting only on $b$.
It then follows from Lemma~\ref{le:kerrewrite} that
\[\gamma_1(\alpha_{\beta([w])}^{-1}\beta\alpha_{[w]})\gamma_2(\alpha_{[w]}^{-1}\beta^{-1}\alpha_{\beta([w])})=\gamma_1(\alpha_{[w]}\alpha')\gamma'(\alpha_{[w]}\alpha')^{-1}\in\langle S_k\rangle\]
which proves the lemma in this case.

If $\alpha=\alpha_{(C,c)([w])}^{-1}(C,c)\alpha_{[w]}$,
then the lemma follows from Lemma~\ref{le:cinQrewrite} if $\pg{c}=\pg{b}$ and from Lemma~\ref{le:rewritegens} if $\pg{c}\neq\pg{b}$.
\end{proof}

\begin{proof}[Proof of Proposition~\ref{pr:1ofeach}]
For each $b\in\supp Q$, take $S(b)$ to be the subgroup generated by $\{\tau_{a,b}|a\in X, a\geq b\}$ and the inversion with respect to $b$, and take
\[S=\bigcup_{b\in\supp Q}S(b).\]
For $\alpha\in(\AAGo)_{[w_0]}$, take $d(\alpha)$ to be the minimum number of elements of $S$ appearing in any factorization of $\alpha$ as a product of elements of $(S_k\cup S_e\cup S)^{\pm1}$.
Note $S_i\subset S$, so such a factorization exists by Proposition~\ref{pr:Sektgens}.

We will prove this proposition by induction on $d(\alpha)$.
If $d(\alpha)=0$, then the proposition is obviously true.
Now suppose that we have $\alpha=\beta\gamma\alpha'$, where $\gamma\in S(b)$ for some $b\in\supp Q$,  $\beta\in \langle S_e\cup S_k\rangle$, and $\alpha'\in(\AAGo)_{[w_0]}$ with $d(\alpha')=d(\alpha)-1$.
By repeated application of Lemma~\ref{le:across1}, we know we have some $b'\in \supp Q$, some $\gamma'\in S(b')$, and some $\beta'\in\langle S_k\cup S_e\rangle$ with $\beta\gamma=\gamma'\beta'$.
Then $\alpha=\gamma'\beta'\alpha'$.
Since $d(\beta'\alpha')\leq d(\alpha)-1$, we can apply the inductive hypothesis and get $\gamma''\in\langle S_i\rangle$ and $\beta''\in\langle S_k\cup S_e\rangle$ with $\beta'\alpha'=\gamma''\beta''$.
Then $\alpha=\gamma'\gamma''\beta''$; since $\gamma'\gamma''\in\langle S_i\rangle$, we are done.
\end{proof}

\begin{lemma}\label{le:SkeQ}
We have $\langle S_k\cup S_e\rangle < (\AAGo)_Q$.
\end{lemma}

\begin{proof}
For $\delta\in S_k$, since $\delta\in\Ker(\AAGo\to\Aut H_\Gamma)$, it is obvious that $\delta_*Q=Q$.

For $[w]\in T'$, note that $(\alpha_{[w]})_*Q=Q$.
This is because $\alpha_{[w]}$ is a product of elements $(A,a)$ with $\trans(A,a)\subset\supp w_0$.
If $\alpha=\alpha_{\beta([w])}^{-1}\beta\alpha_{[w]}\in S_e$, then either $\beta$ is a permutation fixing $(\supp Q)^{\pm1}$ or $\beta$ is a non-permutation Whitehead automorphism with $\trans \beta\subset \supp w_0$.
In either case, it follows that $\alpha_*Q=Q$.
\end{proof}

\begin{theorem}\label{th:almostthere}
The group $(\AAGo)_{([w_0],Q)}$ is generated by the finite set $S_e\cup S_k\cup S_Q$.
\end{theorem}

\begin{proof}
As previously noted, $S_e$ and $S_k$ are finite.
The set $S_Q$ is finite because $Q$ is finite.
Now suppose that $\alpha\in(\AAGo)_{([w_0],Q)}$.
By Proposition~\ref{pr:1ofeach}, we can rewrite $\alpha$ as $\beta\gamma$ where $\beta\in\langle S_e\cup S_k\rangle$ and $\gamma\in \langle S_i\rangle$.

By Lemma~\ref{le:SkeQ}, we know that $\beta_*Q=Q$.
Since $\alpha_*Q=Q$, it follows that $\gamma_*Q=Q$.
Then by Theorem~\ref{th:domspred}, there is an element $\delta\in\langle S_Q\rangle$ such that $\gamma\delta^{-1}\in\Ker(\AAGo\to\Aut H_\Gamma)$.
Then by Proposition~\ref{pr:KZisker} (with $Z=\supp Q$), we know that $\gamma\delta^{-1}\in\langle S_k\rangle$.
Since $\alpha=\beta(\gamma\delta^{-1})\delta$, we have proven the theorem.
\end{proof}

\begin{proposition}\label{pr:nonought}
The group $(\AAG)_{([w_0],Q)}$ is finitely generated.
\end{proposition}

\begin{proof}
Recall that $\AAGo$ is a finite-index normal subgroup of $\AAG$.
Then $\AAGo$ is also finite-index and normal in $\langle \AAGo,(\AAG)_{([w_0],Q)}\rangle$.
By the classical second isomorphism theorem, we have:
\begin{equation*}\begin{split}
\langle &\AAGo,(\AAG)_{([w_0],Q)}\rangle/\AAGo\\
&\cong (\AAG)_{([w_0],Q)}/\big((\AAG)_{([w_0],Q)}\cap\AAGo\big)
\end{split}
\end{equation*}
But $(\AAG)_{([w_0],Q)}\cap\AAGo=(\AAGo)_{([w_0],Q)}$, so $(\AAGo)_{([w_0],Q)}$ is finite-index in $(\AAG)_{([w_0],Q)}$.
So we are done by Theorem~\ref{th:almostthere}.
\end{proof}

\begin{proof}[Final step in the proof of Theorem~\ref{mt:finitelygenerated}]
Let $Z_{A_\Gamma}(w_0)$ denote the centralizer of $w_0$ in $A_\Gamma$.
Consider the following sequence of maps, which we will show to be exact:
\[Z_{A_\Gamma}(w_0)\to\Mod(\Gamma,w_0,Q)\stackrel{\pi}{\to}(\OAG)_{([w_0],Q)}\to 0\]
Here the first map is the map sending an element to its corresponding inner automorphism.

If $[\alpha]\in(\OAG)_{([w_0],Q)}$ and $\alpha\in\AAG$ is a lift of $\alpha$, then $\alpha$ sends $w_0$ to a conjugate $u^{-1}w_0u$.
If we compose $\alpha$ with the inner automorphism given by conjugation by $u^{-1}$, we get an automorphism in $\Mod(\Gamma,w_0,Q)$ that projects to $[\alpha]$.
This explains the surjectivity of $\pi$.

If $\alpha$ is in the kernel of $\pi$, then it is the inner automorphism $c_u$ for some $u\in A_\Gamma$.
Of course, $c_u\in\Mod(\Gamma,w_0,Q)$ if and only if $c_u\in(\AAG)_{w_0}$, which is true if and only if $u\in Z_{A_\Gamma}(w_0)$, proving the exactness of this sequence.

From Proposition~\ref{pr:nonought}, the group $(\OAG)_{([w_0],Q)}$ is finitely generated.
Servatius's centralizer theorem from~\cite{se} completely describes the centralizers of elements in $A_\Gamma$; in particular, it tells us that $Z_{A_\Gamma}(w_0)$ is finitely generated.
Since $\Mod(\Gamma,w_0,Q)$ surjects onto a finitely generated group with finitely generated kernel, it is finitely generated.
\end{proof}

\section{Closing Remarks}\label{se:conclude}
The work in this paper opens the way for further study of mapping class groups over graphs.
First of all, it would be interesting to recover Definition~\ref{de:sympstruct} by means of a geometric construction.
In the extreme cases, $\Sp(2g,\Z)$ can be seen as the linear automorphisms of the torus $\mathbb{T}^{2g}$ that preserve a standard symplectic differential form, and $\Mod_{g,1}$ can be seen as the homotopy group of self-homotopy-equivalences of a graph that preserve some additional combinatorial structure called a ``fat graph" structure (see Penner~\cite{penner}).
It is worth noting that $\mathbb{T}^{2g}$ and certain graphs are examples of \emph{Salvetti complexes}.
The Salvetti complex $S_\Gamma$ is a finite cubical complex that forms a natural $K(A_\Gamma,1)$ space (see Definition~2.6 in Charney~\cite{char}).
If we take the monoid of self-homotopy-equivalences of $S_\Gamma$ and take a quotient by considering maps equivalent if they are homotopic, we get a group.
Call this group $G$. 
Note that $G\cong \AAG$.
This brings us to the following problem:
\begin{problem}\label{prob:geomdef}
Produce a structure on $S_\Gamma$ and a corresponding symplectic structure $(w,Q)$ on $A_\Gamma$ such that the subgroup of $G$ of elements represented by maps fixing this structure is naturally isomorphic to $\Mod(\Gamma,w,Q)$.
\end{problem}
By a structure on $S_\Gamma$, I mean some extra combinatorial data, or some extra differential data, or some combination of the two.

Theorem~\ref{mt:finitelygenerated} could be a starting point for future homological finiteness results about $\Mod(\Gamma,w,Q)$.
This conjecture could be a possible next step.
\begin{conjecture}\label{con:finpres}
For every graph $\Gamma$ with a symplectic structure $(w,Q)$ on $A_\Gamma$, the group $\Mod(\Gamma,w,Q)$ is finitely presented.
\end{conjecture}
There are combinatorial methods to show that $\Mod_{g,1}$ is finitely presented (see McCool~\cite{mccool}) which could potentially be extended to prove Conjecture~\ref{con:finpres}.

In the extreme cases, it is known that both $\Mod_{g,1}$ and $\Sp(2g,\Z)$ contain finite index subgroups with finite $K(\pi,1)$ complexes.
This implies that both groups are of type \emph{VFL}, which is a strong homological finiteness condition (see Brown~\cite{brown}, chapter VIII.11).
This leads us to the following conjecture.
\begin{conjecture}\label{con:vfcx}
For every graph $\Gamma$ with symplectic structure $(w,Q)$ on $A_\Gamma$, the group $\Mod(\Gamma,w,Q)$ has a finite-index subgroup $G$ with a finite $K(G,1)$ complex.
\end{conjecture}
It seems unlikely that Conjecture~\ref{con:vfcx} could be proven by purely combinatorial methods, but given a solution to Problem~\ref{prob:geomdef}, it is conceivable that one could recover such a $K(G,1)$ complex as a kind of moduli space of Salvetti complexes with symplectic structures.
A related problem would then be to find bounds on the virtual cohomological dimension of $\Mod(\Gamma,w,Q)$.
Charney--Crisp--Vogtmann~\cite{ccv} and Charney--Vogtmann~\cite{cv} have already made much progress on the parallel problem for $\AAG$.

Theorem~\ref{mt:magnusthm} could be a starting point for work on the homological properties of $\IAut A_\Gamma$.
In the usual way (as with $\Mod_{g,1}$ or with $\Aut F_n$), the action of $\AAG$ on the 2-step nilpotent truncation of $A_\Gamma$ defines an $\AAG$--equivariant homomorphism (a \emph{Johnson homomorphism}) from $\IAut A_\Gamma$ to an abelian group.
We can then ask the following question.
\begin{question}
For arbitrary $\Gamma$, is the image of the Johnson homomorphism on $\IAut A_\Gamma$ equal to the abelianization of $\IAut A_\Gamma$?
\end{question}
This question was answered in the affirmative for $IA_n$, independently by Cohen--Pakianathan, by Farb, and by Kawazumi~\cite{kawazumi} (see Theorem~1.1 of Pettet~\cite{pettet}).

The following conjecture is linked to Charney--Vogtmann~\cite{cv}.
\begin{conjecture}
For every graph $\Gamma$, the group $\IAut A_\Gamma$ is torsion-free and there is a finite-dimensional $K(\IAut A_\Gamma,1)$ complex.
\end{conjecture}
A related problem is to bound the dimension of such a complex, as Bestvina--Bux--Margalit~\cite{bbm} did in the case of $IA_n$.
We do not expect such a complex to have finitely many cells in each dimension, but only that such a complex would be finite-dimensional.

\bibliographystyle{amsplain}
\bibliography{ssraag}

\noindent
Dept. of Mathematics, California Institute of Technology\\
Pasadena, Ca 91125\\
E-mail: {\tt mattday@caltech.edu}
\medskip

\end{document}